\documentclass{article}[12pt]
\usepackage{quiver}
\usepackage{bussproofs}
\usepackage[utf8]{inputenc}
\usepackage{tensor}
\usepackage{mathrsfs}
\usepackage{graphicx}
\usepackage{cite}
\usepackage{amsmath}
\usepackage[all,cmtip,2cell]{xy}
\usepackage{mathrsfs}
\usepackage{hyperref}
\usepackage{lineno}
\usepackage{mathrsfs}
\usepackage{amsfonts}
\usepackage{amsthm}
\usepackage{amssymb}
\usepackage{mathtools}
\usepackage{stmaryrd}
\usepackage{xcolor}
\usepackage[letterpaper, margin=1in]{geometry}
\usepackage[capitalize,noabbrev,nameinlink]{cleveref}
\numberwithin{equation}{section}
\theoremstyle{definition}
\newtheorem{define}{Definition}[section]
\newtheorem{example}[define]{Example}
\newtheorem{construction}{Construction}[section]
\theoremstyle{remark}
\newtheorem{remark}[define]{Remark}
\theoremstyle{plain}
\newtheorem{theo}[define]{Theorem}
\newtheorem{lemma}[define]{Lemma}
\newtheorem{prop}[define]{Proposition}
\newtheorem{cor}[define]{Corollary}

\newcommand{\E}{\mathscr E}

\newcommand{\upset}{\uparrow \!\!}
\newcommand{\downset}{\downarrow\!\!}

\newcommand{\slat}{\mathbf{sLat}}
\newcommand{\interior}{\mathsf{int}}
\newcommand{\closure}{\mathsf{cl}}
\newcommand{\mspec}{\mathbf{mSpec}}

\newcommand{\mscr}{\mathscr}

\newcommand{\set}{\mathbf{Set}}

\makeatother
\makeatletter
\def\slashedarrowfill@#1#2#3#4#5{%
  $\m@th\thickmuskip0mu\medmuskip\thickmuskip\thinmuskip\thickmuskip
  \relax#5#1\mkern-7mu%
  \cleaders\hbox{$#5\mkern-2mu#2\mkern-2mu$}\hfill
  \mathclap{#3}\mathclap{#2}%
  \cleaders\hbox{$#5\mkern-2mu#2\mkern-2mu$}\hfill
  \mkern-7mu#4$%
}
\def\rightslashedarrowfill@{%
  \slashedarrowfill@\relbar\relbar\mapstochar\rightarrow}
\newcommand\xslashedrightarrow[2][]{%
  \ext@arrow 0055{\rightslashedarrowfill@}{#1}{#2}}
\makeatother
\def\hto{\xslashedrightarrow{}}

\title{A Topos-Theoretic Semantics of Intuitionistic Modal Logic with an Application to the Logic of Branching Spacetime}
\author{Michael J. Lambert}
\date{October 2024}

\begin{document}

\maketitle

\begin{abstract}
    The Alexandrov topology affords a well-known semantics of modal necessity 
    and possibility. This paper develops an Alexandrov topological semantics of 
    intuitionistic propositional modal logic internally in any elementary 
    topos. This is done by constructing \emph{interior} and \emph{closure} 
    operators on the power-object associated to a given relation in the ambient 
    topos. When the relation is an order, these operators model intuitionistic 
    S4; when the relation is an equivalence relation, they also model the 
    characteristic (B) axiom of classical S5. The running example of interest 
    arises from the \emph{Branching space-time} of Nuel Belnap, which is shown 
    to induce a \emph{histories presheaf} upon which can be defined an 
    equivalence relation of being \emph{obviously undivided} at a given point 
    event. These results have some philosophical implications. For example, we 
    study the branching space-time example in light of the 
    \emph{indistinguishability interpretation} of epistemic modal logic. We 
    will also study several famous first-order formulas in presheaf topos 
    semantics such as the so-called \emph{Barcan formula}. We shall see, 
    however, that one of the Barcan converses is invalidated by a simple 
    example of non-trivial space-time branching. This invalidates a thesis of 
    \textit{metaphysical actualism}, namely, that there are no possibly 
    existing but non-actual entities.
\end{abstract}

\tableofcontents

\section{Introduction}

Topos-theoretic developments of \emph{modality} come in two flavors. 
On the one hand, there are the \emph{geometric modalities} manifesting as local 
operators and leading to the theory of sites and sheaves. On the other hand, 
there are various internalized incarnations of the usual modal operators 
for necessity and possibility. These two have in common the goal of defining or 
constructing certain morphisms in the topos structure that model the behavior 
of the modality in question. Although local operators are well-studied, the 
topos-theoretic accounts of necessity and possibility are fewer and less 
standardized in their approaches. Perhaps this partly is owing to the nature of 
modal logic itself. That is, the modal logic of possibility and necessity is 
usually a classical, propositional logic, presented in a Hilbert-style 
formalism, while the internal logic of toposes is in general intuitionistic, 
higher-order, and presented in a sequent-calculus or type-theoretic formalism. 
Consequently, there is not a great deal of consensus about what exactly 
intuitionistic modal logic should be, much less a great deal of systematization 
of, or even interest in, its categorical semantics.

In the interest of supplying this lack, this paper is a contribution to the 
subject of the categorical semantics of intuitionistic modal logic. The goal is 
to show how to internalize the ordinary semantics of the Alexandrov topology in 
an arbitrary topos. This is done by starting with an internal relation and then 
developing internalized interior and closure operators on the power-object 
associated to the underlying object of the relation. When the relation is an 
order, the power-object models intuitionistic S4; when the relation is also 
symmetric, it models the (B) axiom which would be characteristic of 
intuitionistic S5. These results are purely topos-theoretic analogues of the 
standard Kripke semantics of modal logic \cite{kripke1963} phrased in terms of 
the operators of topological semantics \cite{mckinseytarski1944} associated to 
orders. Although this is a sort of \emph{back to basics} approach that may 
appear outwardly naive or too simplistic, we believe that there are a variety 
of useful constructions that come up along the way and that the overall 
development illustrates well the use of Kripke-Joyal forcing semantics. There 
are also applications and philosophical implications previewed below.

Investigation of the topos forcing semantics of modal operators is not exactly 
new, but to our knowledge only three studies, namely, 
\cite{awodeykishidakotzsch2014}, \cite{reyeszolfaghari1991} and 
\cite{reyes1991} have done so in any depth. These have in common the 
construction of certain \emph{internal adjunctions} used to interpret the 
necessity and possibility operators. In particular, all three studies interpret 
necessity by a comonad associated to an internal frame homomorphism arising 
from a given external geometric morphism either to or from the ambient topos in 
question. Although such subsidiary geometric morphisms usually exist, or at 
least can be cooked up, it seems worthwhile to develop a forcing semantics 
without the assumption of any such geometric morphism. Such an account is 
completely internal using only the given structure of the single fixed topos 
viewed as a self-contained \emph{universe of discourse}. The present attempt to 
do so derives the required internal adjunctions from the topos structure 
directly using only the assumption of the existence of an internal relation. 
Such relations abound; for example, we may take the diagonal on a given object. 
Our internalization of these topological semantics in the form of the 
Alexandrov topology also has the satisfying outcome that it is intuitive and 
formally very close to the ordinary set-theoretic version.

This paper is not one on the syntax or proof theory of modal logic. We do not 
attempt to provide novel or state-of-the-art modal systems. For intuitionistic S4 we are happy simply to adopt the natural deduction system (IS4) of 
\cite{biermandepaiva2000} which has the satisfying feature of 
\emph{substitutivity of proofs}. This system has a ready-made, sound categorical semantics, and our models will be shown to be of this type. When it comes to intuitionistic S5, there appears to be no corresponding system of natural deduction satisfying substitutivity of proofs. However, the distinguishing axiom of classical S5 is the (B) axiom which is modelled categorically by the MAO-couples of \cite{reyes1991} and \cite{reyeszolfaghari1991}. For this reason we skirt the question of giving a system (IS5) but think of MAO-couples as potential models of such a system, were it to be developed in subsequent future work. One main result is that any internal equivalence relation affords canonically an MAO-couple, hence a model of some hypothetical system (IS5).

Our application is to the logic of the so-called \emph{branching space-time} \cite{belnap1992}. This turns out to induce canonically an MAO-couple in a certain presheaf topos. This model stems from what we call the \emph{histories presheaf}, defined on the underlying poset of branching space-time. The histories presheaf sends each point event of branching space-time to the set of all histories of that point event. The relation \emph{obviously undivided} at a point event induces an equivalence relation on the histories presheaf, hence an MAO-couple. This model gives an especially satisfying gloss on modal statements with a distinctly epistemic flavor, which we envision as being relevant to formulating categorical epistemic modal logic generally. The way that this is achieved is by looking at the Kripke-Joyal forcing semantics of modal statements arising from the internal interior and closure operators associated to our Alexandrov topology constructions.

A simple example of non-trivial branching figures prominently in a concluding application to the analysis of several first-order formulas. In particular, we consider the so-called \textit{Barcan formulas} and their converses \cite{barcan_1947} which seem to be a perennial source of discussion in formal metaphysics and epistemology. Our analysis of these formulas is applied to existence predicates \cite{scott1979} and briefly to the issue of \textit{actualism} versus \textit{possibilism} in the metaphysics of modality (stemming from the the introduction of possible world semantics to explain modality \cite{carnap1947}, \cite{hintikka1962}, \cite{kripke1963}, \cite{lewis1986}). Against the possibilist position that there exist non-actual entities, actualism asserts that there are no possibly existing but non-actual entities \cite{adams1974}. We take the \emph{actualist thesis} in a positive formulation to be 
    \begin{equation*} \label{equation:actualist-thesis}
        \text{what possibly exists is actual} \tag{A}
    \end{equation*}
i.e. there are no merely possible but non-actual entities. Perhaps 
unsurprisingly, our result is that non-trivial branching invalidates this 
thesis. The reason is that non-trivial branching means that alternate histories 
of a point event exist and are possible but not necessarily actual. The point 
of our development is that the invalidation of the actualist thesis is 
deducible from the introduced formalism of topos-theoretic modality and 
non-trivial branching. The question as to whether this in any way settles the 
debate turns upon whether branching space-time is a sound model and is not 
taken up in this paper.

In more detail, some of our principal results and contributions are the following.

\begin{enumerate}
    \item Interior and closure operators associated to an internal relation 
    appear in \cref{definition:internalinterioroperator} and 
    \cref{definition:internalclosureoperator} using certain naturally induced 
    morphisms on power-objects in the ambient topos. The interior operator is a 
    composition of a downset operator followed by an inverse image. Likewise, 
    the closure operator is a direct image followed by a join on the associated 
    power-object.
    \item The forcing semantics of these operators are the central ingredients 
    of our interpretation of modal necessity and possibility. The resulting 
    semantics appear in \cref{proposition:interiorsemantics} and 
    \cref{corollary:closureforcingsemantics}. These forcing statements mimic 
    the ordinary set theoretic topological semantics of these operators in the 
    case of the Alexandrov topology but incorporate the characteristic forcing 
    element of \emph{being true at a stage}.
    \item Along the way, in \cref{lemma:forcingsemanticsofdirectimage}, we 
    clarify the forcing semantics of direct images in toposes. We also give a 
    forcing semantics of the join operator on a given internal power-object in 
    \cref{proposition:forcingsemanticsofinternaljoinoperator}. Together these 
    give our semantics of the Alexandrov closure, discussed immediately above.
    \item The main theoretical results of the paper are 
    \cref{theorem:ordermodels(IS4)} and 
    \cref{theorem:equivalencerelationmodels(IS5)} which show that under 
    suitable conditions on the underlying relation, the associated power-object 
    with its interior and closure operators are categorical models of necessity 
    and possibility.
    \item The example of branching space-time \cite{belnap1992}, a synthetic 
    fusion of relativity and algebraic semantics of branching temporal logic, 
    is shown to induce canonically an equivalence relation in a certain 
    presheaf 
    topos, hence an MAO-couple, in \cref{theorem:historiespresheafmodelsIS5}.
    \item There appears to be some general philosophical import of our 
    constructions. This has its origins in unpacking the forcing semantics of 
    the modal operators in the context of presheaf categories and particularly 
    the example of branching space-time. The purpose of 
    \cref{subsection:indistinguishability} is to make an analogy the 
    \emph{indistinguishability interpretation} of necessity statement in modal 
    epistemic logic.
    \item The concluding discussion \cref{subsection:firstorderconsiderations} 
    exhibits several first-order modal formulas for analysis. We shall see that 
    equivalence relations in toposes in particular always validate the 
    so-called \emph{Barcan formula}, but by example that the converse does not 
    hold. Likewise, we shall see that \emph{de re} existential statements 
    always imply \emph{de dicto} existential statements but again by example 
    that the converse is not true. In the context of existence predicates, this 
    result thus says that whatever exists necessarily exists, but not 
    conversely. For this reason, branching space-time is not a model of 
    \emph{metaphysical actualism}.
\end{enumerate}

The plan of the paper is as follows: \cref{section:intuitionisticmodallogic} 
gives an overview of our version of intuitionistic S4 and S5, recounts the 
usual set-theoretic topological semantics of their classical versions, provides 
an overview of basic topos theory and internal preorder theory in the form in 
which we need it; the culmination is the definition of the topos-theoretic 
algebraic models of our intuitionistic modal systems; 
\cref{section:internalizedalexandrovtopology} presents our internalization of 
the Alexandrov topology, including our construction of the interior and closure 
operators and prove that they provide the appropriate algebraic models of 
intuitionistic S4 and S5; finally \cref{section:branchingspacetime} shows that 
the example of branching space-time induces a canonical equivalence relation on 
the histories presheaf and thus as a result of the general theory affords a 
model of intuitionistic S5.

\paragraph{Acknowledgements} Questions posed by Kohei Kishida and Marcello 
Lafrancesco led to the undertaking of this project in earnest. Special thanks 
are due to Valeria de Paiva who upon reading a draft of the manuscript provided 
indispensible guidance on the categorical semantics of intuitionistic modal logic.

\section{Topos-Theoretic Algebraic Semantics}
\label{section:intuitionisticmodallogic}

This section is mostly expository in nature. It summarizes the adopted system 
(IS4) and after reviewing some topos theory gives a definition of the internalized 
models of (IS4) and the models of the hypothetical system (IS5) in the form of 
MAO-couples. We provide some subsidiary results and clarifications along the way 
that will help in subsequent developments.

\subsection{The System (IS4)}
\label{subsection:SetTheoreticSemanticsofModalLogicOverview}

Consideration of the connections between space-time and modal logic probably 
originate with \cite{goldblatt1980}. The modal system there is however classical and since we will be working in toposes ours ought to be \emph{intuitionistic}. However, in dealing with modal logic, our purpose is neither to give a comprehensive presentation of the subject nor to exhibit a state-of-the-art system. It is rather merely to present a reasonable, minimal system to open the door to topos-interpretations of possibility and necessity with a minimal assumption of available structure. In particular, we do not assume the existence of a base topos, or any extra external geometric morphisms as in \cite{reyes1991} and \cite{awodeykishidakotzsch2014}, although these are usually available.

The system of choice is \textbf{intuitionistic S4} (IS4) of 
\cite{biermandepaiva2000}. Specifically, we assume throughout an intuitionistic propositional logic $\bot$, $\top$, $\neg \phi$, $\phi\wedge \psi$, $\phi\vee \psi$, $\phi\supset \psi$, augmented by modal necessity $\Box \phi$ and possibility $\Diamond \phi$ operators subject to the axioms of the usual intuitionistic calculus (IPC) plus the further modal axioms, namely,
    \begin{enumerate}
        \item $\Box(\phi\supset \psi)\supset (\Box \phi\supset \Box \psi)$
        \item $\Box \phi\supset \phi$ and $\phi\supset\Diamond \phi$
        \item $\Box \phi\supset \Box\Box \phi$
        \item $\Box (\phi\supset \Diamond \psi)\supset (\Diamond \phi \supset \Diamond \psi)$.
    \end{enumerate}
The natural deduction calculus for (IS4) is detailed in 
\cite[\S 4]{biermandepaiva2000} where it is shown to satisfy 
\emph{substitutivity of proofs}, a feature missing from previous accounts of 
intuitionistic S4 (e.g. \cite[Chapter VI]{prawitz2006}).
Our choice of this system is due to the fact that \cite{biermandepaiva2000} 
provides a sound categorical semantics of (IS4) in the form of certain 
distinguished \emph{categorical models}. This is reviewed and internalized in 
\cref{def:categorical-model-IS4}. A main result of this paper 
(\cref{theorem:ordermodels(IS4)}) is that any order relation in a topos affords 
a categorical model of (IS4) in this sense.

Our examples prompt consideration of certain additional modal axioms. The 
classical system \textbf{Diodorean S4} of \cite{goldblatt1980} is also the 
modal system of knowledge and belief of \cite{stalnaker2006}. This is built on 
classical propositional logic and includes the comonad axioms for $\Box$ 
(the monad axioms for $\Diamond$ follow classically from inter-definability), 
along with the \emph{distributive law}, namely, $\Diamond\Box \phi\supset 
\Box\Diamond \phi$. This modal axiom is sometimes labeled (C) since the 
corresponding frame condition is a convergence law on the associated relation, 
saying that two elements with a common lower bound also have a common upper 
bound:
    \begin{equation*}
        x\leq y \text{ and } x\leq z \text{  implies  } y\leq w \text{ and } z\leq w \text{ for some } w.
    \end{equation*}
The modal system (IS4) augmented by the distributive law (C) will be called 
\textbf{distributive (IS4)}. The categorical semantics of (IS4) and 
distributive (IS4) is discussed in \cref{subsection:toposmodelsgenerally} and 
specifically \cref{def:categorical-model-IS4} where we adopt the categorical 
models of \cite{biermandepaiva2000}.

The example of \emph{branching space-time} \cite{belnap1992} models a more 
specific modal system that we are tempted to call \textbf{intuitionistic S5} 
(IS5). This would hypothetically be a variation on (IS4) above plus the 
further axiom scheme (B)
    \begin{equation} \label{equation:s5axiomB}
        \phi\supset\Box\Diamond \phi \text{  and  } \Diamond\Box \phi\supset \phi.
    \end{equation}
which is the characteristic scheme added to classical S4 to obtain classical S5.
However, although there are several proposal for a system (IS5) (e.g.\cite{prawitz2006}, \cite{simpson1994}, \cite{alechina2001}) no such system with a natural deduction calculus satisfying substitutivity of proofs appears to exist. For this reason we bracket the question of presenting a system (IS5) but think of some hypothetical system as minimally (IS4) plus (B). Our preference is to switch to an existing categorical semantics and take the MAO-couples (\cref{def:MAO-couple}) of \cite{reyes1991} and \cite{reyeszolfaghari1991} as likely candidates for a class of models of such a hypothetical system. This is because MAO-couples are preorder categorical models of (IS4) 
(\cref{lemma:MAO-are-models-IS4}) for which the relevant structure forms 
an internal adjoint modality modeling the distinguishing (B) axiom. One of our 
main results (\cref{theorem:equivalencerelationmodels(IS5)}) is that any 
equivalence relation in a topos $\E$ induces in a canonical way an MAO-couple, 
hence what we expect will be a categorical model of a system (IS5).

Our models are based on the traditional semantics of classical S4 originating with \cite{kripke1963}, \cite{kripke1965}. Classical S4 is interpreted via an \emph{accessibility relation} $R$ on a set $X$ that is at least reflexive and transitive. We shall refer to such a relation as an \emph{order}. Accessibility relations interpreting classical S5 are equivalence relations, that is, those relations that are reflexive, transitive and also symmetric. The related \emph{topological semantics} originate with \cite{mckinseytarski1944} and are summarized in \cite[\S 1.3]{awodeykishida2008}, \cite[\S 2.2]{kishida2011} and \cite[\S 3.2]{kishida2017} where it is shown in particular how topological semantics recover and generalize those of Kripke. In this set up, let $(X,\mscr O(X))$ denote a topological space with $\mscr O(X)$ the frame of open sets. Each proposition is interpreted as a subset $\llbracket \phi \rrbracket\subset X$ in such a way that
    \begin{enumerate}
        \item $\llbracket \neg \phi \rrbracket = X\setminus \llbracket \phi\rrbracket$
        \item $\llbracket \phi\wedge \psi \rrbracket = \llbracket \phi\rrbracket \cap \llbracket \psi\rrbracket$
        \item $\llbracket \phi\vee \psi\rrbracket = \llbracket \phi\rrbracket \cup \llbracket \psi\rrbracket$
        \item $\llbracket \phi\supset \psi\rrbracket = \llbracket \phi\rrbracket \multimap \llbracket \psi\rrbracket$
    \end{enumerate}
where `$\multimap$' here denotes the implication operator on the Boolean algebra powerset $PX$. The box operator $\Box \phi$ is interpreted using the associated interior operator
    \begin{equation*}
        \interior\colon PX\to\mscr O(X)\qquad\qquad S\mapsto \bigcup\lbrace U\in \mscr O(X)\mid U\subset S\rbrace
    \end{equation*}
taking a subset $S\subset X$ to the union of all open sets contained in $S$. 
That is, declare 
    \begin{equation*}
        \llbracket \Box \phi \rrbracket := \interior\llbracket \phi\rrbracket.
    \end{equation*}
Likewise, $\Diamond$ is interpreted using the closure operator $\closure\colon 
PX\to \mscr O(X)$, namely, the intersection of all closed sets containing a given set. The usual semantics is that the sequent $\phi\vdash\psi$ is valid if, and only if, $\llbracket \phi\rrbracket \subset \llbracket\psi\rrbracket$ holds in $PX$.

A particular example is the so-called \emph{Alexandrov topology} associated to 
any set with a relation \cite[\S II,1.8]{johnstone1982}. Let $(X,\leq)$ denote 
an order and declare the opens to be the upsets: those $U\subset X$ such that 
$x\in U$ and $x\leq y$ implies $y\in U$ too. Consequently, the closed sets are 
the \emph{downsets}. We then have the following element-wise descriptions of 
closure and interior:
    \begin{equation} \label{equation:interiorAlexandrovTopologyElementDescription}
        s\in \interior (S) \text{ if, and only if, } s\leq x \text{ implies } x\in S
    \end{equation}
and
    \begin{equation} \label{equation:closureAlexandrovTopologyElementDescription}
        s\in\closure (S) \text{ if, and only if, } s\leq x \text{ for some } x\in S. 
    \end{equation}
Thus, interpreting $\Box$ as $\interior$, we have the Alexandrov semantics of 
necessity, namely, 
    \begin{align*}
        s\in \llbracket \Box p \rrbracket & \text{ if, and only if, } s\in \interior \llbracket p\rrbracket \\
                                          & \text{ if, and only if, } x\in \llbracket p \rrbracket \text{ for all $x$ such that } s\leq x.  
    \end{align*}
Likewise, interpreting the possibility operator $\Diamond$ as Alexandrov closure, 
we have the existential semantics
    \begin{align*}
        s\in \llbracket\Diamond p \rrbracket & \text{ if, and only if, } s\in \closure \llbracket p\rrbracket \\
                                             & \text{ if, and only if, } x\in \llbracket p \rrbracket \text{ for some $x$ such that } s\leq x
    \end{align*}
Thinking of a proposition $S\subset X$ being true for an individual $x\in X$ if, and only if, $x\in S$, this reproduces the original semantics of $\Box$ given in \cite{kripke1963}. The goal of this paper is to construct an internal Alexandrov topology for any object with a relation in an arbitrary topos. To this end, we need to construct associated interior and closure operators and to better understand their forcing semantics. This will be done in 
\cref{subsection:interiorandclosureinternally} and \cref{subsection:operatorsemantics} 
below.

\subsection{Toposes and Their Internal Logic}
\label{subsection:toposesandinternallogic}

Since our goal is to internalize the Alexandrov topology, we must set out the 
main features of elementary toposes that will figure in our development. Our
reference is \cite{maclanemoerdijk1992}. Briefly, a topos is a finitely 
complete category $\mscr E$ with a subobject classifier $\Omega$ and 
power-objects.

\begin{define} \label{definition:power-object}
    The \textbf{power object} of an object $X\in\mscr E$ is an object $PX$ 
    together with a membership morphism $\in_X\colon X\times PX\to \Omega$ with 
    the property that for any $f\colon X\times Y\to \Omega$ there is a unique 
    morphism $\hat f$ making
        \[\begin{tikzcd}
            {X\times PX} & \Omega \\
            {X\times Y}
            \arrow["{\in_X}", from=1-1, to=1-2]
            \arrow["{1\times\hat f}", from=2-1, to=1-1]
            \arrow["f"', curve={height=6pt}, from=2-1, to=1-2]
        \end{tikzcd}\]
    commute. Call $\hat f$ the \textbf{transpose} of $f$. \qed
\end{define}

By uniqueness, for any such morphism $f$, we have $f = 1_X\in_X\hat f$. As a 
special case, we recover $P1\cong \Omega$. 

\begin{example}
    The categories $\set$, $\mathbf{FinSet}$, and $\set^{\to}$ are all 
    toposes.  \qed
\end{example}

\begin{example} \label{example:presheafandsheafcategoriesaretoposes}
    Recall that a presheaf on $\mscr C$ is a functor $\mscr C^{op}\to\set$. 
    Together with natural transformations as morphisms these comprise a topos 
    of presheaves $[\mscr C^{op},\set]$. Likewise, any category of sheaves 
    (presheaves with a unique amalgamation property) $\mathsf{Sh}(\mscr C,J)$ 
    on a site $(\mscr C,J)$ is a topos. Presheaf toposes are treated in detail 
    in \cref{subsection:presheafcategorycomputations}. \qed
\end{example}

For any topos $\E$, the subobject classifier $\Omega$ is an internal Heyting 
algebra. The binary meet operation $\wedge$ in particular is the classifying 
morphism of $\langle\top,\top\rangle\colon 1\to \Omega\times\Omega$ where 
$\top\colon 1\to \Omega$ is the top element. Write $\top_W\colon W\to \Omega$ 
for the arrow $\top\circ !\colon W\to 1\to \Omega$, thought of, as \emph{true 
relative to $W$}. The internal preorder on $\Omega$ is the equalizer of 
$\wedge$ and the first projection. The Heyting implication $\Rightarrow\colon 
\Omega\times\Omega\to\Omega$ is the classifying morphism of the order 
subobject. The internal equality on $\Omega$ is the classifying morphism of the 
diagonal $\Delta\colon \Omega \to\Omega\times \Omega$. 

By \cite[Theorem IV.8.1]{maclanemoerdijk1992}, each power object $PX$ is also 
an internal Heyting algebra and the assignment 
    \begin{equation*}
        g \colon W\to PX\qquad \mapsto \qquad \in_X(1\times g)
    \end{equation*}
is the unique one inducing a natural isomorphism of external Heyting algebras
    \begin{equation} \label{equation:isomorphismofHeytingalgebras}
        \mscr E(W,PX) \cong \mscr E(W\times X,\Omega)
    \end{equation}
whose reverse correspondence is transposition. In particular, each $PX$ has an 
internal preorder given again as the equalizer of the meet and first projection. 
More on internal preorders following \cref{definition:orders} below.

The power-object operation extends to a functor $P\colon \mscr E^{op}\to
\mscr E$ whose arrow assignment is given by sending a given arrow $f$ to the 
transpose $Pf$ appearing in the diagram
    \[\begin{tikzcd}
        Y && {X\times PX} & \Omega \\
        X && {X\times PY} & {Y\times PY}
        \arrow["{\in_X}", from=1-3, to=1-4]
        \arrow[""{name=0, anchor=center, inner sep=0}, "{1\times Pf}", from=2-3, to=1-3]
        \arrow[""{name=1, anchor=center, inner sep=0}, "f"', from=2-1, to=1-1]
        \arrow["{f\times 1}"', from=2-3, to=2-4]
        \arrow["{\in_Y}"', from=2-4, to=1-4]
        \arrow["\mapsto"{description, pos=0.4}, draw=none, from=1, to=0]
    \end{tikzcd}.\]
Uniqueness makes this a contravariant functor on $\mscr E$. Call $Pf$ the 
\textbf{inverse image} morphism associated to $f$. In ordinary set theory, this 
morphism has the action 
    \begin{equation} \label{equation:settheoreticinverseimagecalculation}
        Pf(T) = \lbrace a\in A\mid f(a)\in T\rbrace
    \end{equation}
which is the usual inverse image morphism.

Associated to each such arrow there is also a \textbf{direct image} 
$\exists_f\colon PX\to PY$. First let $m\colon U_X\to X\times PX$ denote the 
subobject classified by $\in_X$ and then $\exists_f$ is the transpose in the 
diagram on the right:
    \[\begin{tikzcd}
        {U_{X}} & I & 1 && {Y\times PY} & \Omega \\
        & {Y\times PX} & \Omega && {Y\times PX}
        \arrow[tail, from=1-2, to=2-2]
        \arrow["{!}", from=1-2, to=1-3]
        \arrow["\top", from=1-3, to=2-3]
        \arrow["\phi"', from=2-2, to=2-3]
        \arrow[two heads, from=1-1, to=1-2]
        \arrow["{(f\times 1)m}"'{pos=0.2}, curve={height=12pt}, from=1-1, to=2-2]
        \arrow["{\in_Y}", from=1-5, to=1-6]
        \arrow["{1\times\exists_f}", from=2-5, to=1-5]
        \arrow["\phi"', curve={height=12pt}, from=2-5, to=1-6]
    \end{tikzcd}.\]
That is, first classify the image of $(f\times 1)m$ via $\phi$ and then define 
$\exists_f$ to be the transpose of that classifying map. In ordinary set 
theory, this has the effect
    \begin{equation} \label{equation:settheoreticdirectimagecalculation}
        \exists_fS = \lbrace b\in B\mid b=f(a) \text{ for some } a\in A\rbrace.
    \end{equation}
Calculations will depend upon the following result, whose proof is 
straightforward by the constructions of the inverse and direct image morphisms 
above.

\begin{lemma} \label{lemma:calculateinverseanddirectimage}
    For any morphism $f\colon A\to B$,
        \begin{enumerate}
            \item $(a\in_A P(f)(s)) = (f(a)\in_B s)$
            \item $(a\in_A s) = (f(a)\in_{PX}\exists_f(s))$
        \end{enumerate}
    for any element $\langle a,s\rangle\colon W\to A\times PA$.\qed
\end{lemma}

We will have need of some generalities on preordered objects in toposes (cf. 
\cite{johnstone2014}, \cite{maclanemoerdijk1992}, and \cite{wood2004}). By a 
relation $r\colon X\hto Y$ in a topos we mean a monomorphism $r\colon R\to 
X\times Y$, or, equivalently, a pair of jointly monic arrows $R\to X$, $R\to 
Y$. Such relations are ordered by $s\leq r$ if, and only if, $s\colon S\to 
X\times Y$ factors through $r\colon R\to X\times Y$ by a (necessarily) monic 
arrow $S\to R$. Relations $s\colon X\hto Y$ and $r\colon Y\hto Z$ compose by 
first taking a pullback as at left below
    \[\begin{tikzcd}
        {S\times_YR} & R && {S\times_YR} & {X\times Z} \\
        S & Y &&& {s\otimes r}
        \arrow[from=1-1, to=1-2]
        \arrow[from=1-2, to=2-2]
        \arrow[from=1-1, to=2-1]
        \arrow[from=2-1, to=2-2]
        \arrow[from=1-4, to=1-5]
        \arrow[two heads, from=1-4, to=2-5]
        \arrow[tail, from=2-5, to=1-5]
    \end{tikzcd}\]
and then taking an image of the product map formed by the leftmost and 
rightmost projections as on the right above. 

\begin{define} \label{definition:orders}
    An \textbf{internal preorder} is a pair $(A,r)$ consisting of an object $A$ 
    and a relation $r \colon A\hto A$ that is internally reflexive and 
    transitive in the sense that
        \begin{enumerate}
            \item $\Delta_A\leq r$
            \item $r\otimes r\leq r$
        \end{enumerate}
    both hold. These are the \textbf{reflexivity} and \textbf{transitivity} 
    axioms respectively. An \textbf{internal equivalence relation} is a 
    reflexive and transitive relation $r$ on $A$ that is also symmetric in the 
    sense that $r^\dagger\leq r$ where $r^\dagger$ is the reversed relation 
    given by composing $r$ with the canonical isomorphism $X\times X\cong 
    X\times X$ interchanging the factors of the product. \qed
\end{define}

Any internal preorder is an internal category. An internal functor is an 
order-preseving morphism. We will write `$\leq_r$' for the classifying arrow of 
$r$ and `$f\leq_r g$' to mean that the generalized elements $f,g\colon X\to A$ 
factor through the order $r$ as in 
    \[\begin{tikzcd}
        & R & 1 \\
        X & {A\times A} & \Omega
        \arrow["{\leq_r}"', from=2-2, to=2-3]
        \arrow["r", tail, from=1-2, to=2-2]
        \arrow["{!}", from=1-2, to=1-3]
        \arrow["\top", from=1-3, to=2-3]
        \arrow["{\langle f,g\rangle}"', from=2-1, to=2-2]
        \arrow[dashed, from=2-1, to=1-2]
    \end{tikzcd}.\]
Equivalently, this is the statement that $f\leq_rg = \top_X$ holds externally. 
Any internal preorder allows construction of \textit{upsegment} and 
\textit{downsegment} internal functors. These arise as certain transposes:

\begin{define}[\S 5.3 \cite{johnstone2014}] 
    \label{define:downsegmentupsegmentarrows}
    If $(A,r)$ is an internal preorder, the \textbf{downsegment} arrow is the 
    transpose of the classifying map of the order object:
        \[\begin{tikzcd}
            R & 1 && {A\times PA} & \Omega \\
            {A\times A} & \Omega && {A\times A}
            \arrow["{!}", from=1-1, to=1-2]
            \arrow["\top", from=1-2, to=2-2]
            \arrow["r"', tail, from=1-1, to=2-1]
            \arrow["{\leq_r}"', from=2-1, to=2-2]
            \arrow["1\times\downarrow", from=2-4, to=1-4]
            \arrow["{\in_A}", from=1-4, to=1-5]
            \arrow["{\leq_r}"', curve={height=12pt}, from=2-4, to=1-5]
        \end{tikzcd}.\]
    The \textbf{upsegment} arrow $\uparrow\colon L\to PL$ is the transpose of 
    the classifying map of the opposite order $r^\dagger$. \qed
\end{define}

\begin{define} \label{definition:PreorderedAdjunction}
    An \textbf{adjunction} between preordered objects $(P,r)$ and $(Q,s)$ in a 
    topos $\mscr E$ is a pair of morphisms $f\colon P\rightleftarrows Q: g$ 
    satisfying the two usual identities $1\leq_r gf$ and $fg\leq_s 1$. \qed
\end{define}

Both segment arrows are internal functors.

\begin{example} \label{example:internalquantifiersasadjoints}
    For any morphism $f\colon A\to B$, the direct and inverse s are 
    internally adjoint with $\exists_f\dashv Pf$. Any such inverse image also 
    has a right adjoint $\forall_f\colon PA\to PB$. At least when the morphism 
    is the unique one $A\to 1$, this has a nice description as the classifying 
    map of the top element of $PA$. In this case, denote the quantifiers as 
    $\forall_A,\exists_A \colon PA\rightrightarrows \Omega$. For a morphism 
    $\phi\colon X\to \Omega$ thought of as a proposition with a single argument 
    $\phi(x)$ varying over entities of type $X$, define $\forall x\phi(x)$ and 
    $\exists x\phi(x)$ to be the morphisms 
        \begin{equation*}
            \forall_X\hat\phi,\exists_X\hat\phi\colon 1\rightrightarrows \Omega    
        \end{equation*}
    where $\hat\phi$ is the transpose $1\to PX$ of $\phi$.    \qed
\end{example}

Certain structures in toposes are defined by such internal adjunctions. For 
example, we will see that the power-objects $PX$ in any topos are internal 
sup-lattices in the following sense.

\begin{define}[\S 5.3 \cite{johnstone2014}] \label{definition:suplattice}
    An \textbf{internal sup-lattice} is an internal poset $(A,r)$ together with an internal left adjoint $\bigvee\colon PA\to A$ for $\downarrow\colon A\to PA$. Refer to $\bigvee$ as the \textbf{join operator}. Such $(A,r)$ is \textbf{internally cocomplete}. A \textbf{homomorphism} of sup lattices $f\colon (A,r)\to (B,s)$ is a morphism $f\colon A\to B$ that preserves sups in the sense that 
        \[\begin{tikzcd}
            PA & PB \\
            A & B
            \arrow["{\exists_f}", from=1-1, to=1-2]
            \arrow["\bigvee"', from=1-1, to=2-1]
            \arrow["f"', from=2-1, to=2-2]
            \arrow["\bigvee", from=1-2, to=2-2]
        \end{tikzcd}\]
    commutes. Let $\slat(\mscr E)$ denote the category of sup lattices internal to $\mscr E$. \qed
\end{define}

Appropriate compositions of left and right adjoints have the structure of 
internal monads or comonads. Our internal interior and closure operators will 
be examples. The comonad and monad equations are precisely those which validate 
the axioms for necessity and possibility.

\begin{define} \label{definition:comonadmonad}
    A \textbf{comonad} on an preordered object $(X,r)$ is an order-preserving 
    morphism $d\colon X\to X$ satisfying both $d\leq_r 1_X$ and $d\leq_r dd$. 
    The first is the \emph{counit law} and the second is the 
    \emph{comultiplication law}. Dually, a \textbf{monad} $t\colon X\to X$ is 
    an order-preserving morphism satisfying the dual \emph{unit law} 
    $1_X\leq_r c$ and \emph{multiplication law} $cc\leq_r c$.\qed
\end{define}

\subsection{Topos Models of (IS4) and (IS5)}
\label{subsection:toposmodelsgenerally}

As previously discussed, our intention is not to present a novel version of 
intuitionistic modal logic, but merely to adopt a system with an established 
categorical semantics justifying its robustness. To this end, we internalize 
the categorical semantics of (IS4) from 
\cite[Definition 5]{biermandepaiva2000}. There a \emph{categorical model} is 
defined as a cartesian closed category with coproducts together with a monoidal 
comonad and a monad satisfying a strength axiom. Here \emph{monoidal comonad} 
means a lax cartesian monoidal comonad whose structure transformations are also 
cartesian monoidal. This definition involves a number of coherence axioms, all 
of which can be internalized. However, in the interest of readability, we will 
suppress explicit presentation of these since they are automatically validated 
by our internal preorder models. We do take the notion of internal category, 
functor and transformation for granted \cite[Ch.2]{johnstone2014}. 

Throughout fix a topos $\E$.

\begin{define} \label{def:categorical-model-IS4}
    Let $d\colon \mathbb C\to \mathbb C$ denote an internal monoidal comonad on 
    an internal category $\mathbb C$. A \textbf{$d$-strong} monad on 
    $\mathbb C$ is a monad $t$ on $\mathbb C$ together with an internal natural 
    transformation 
        \begin{equation*}
            s\colon d\times t \Rightarrow t(d\times 1)
        \end{equation*}
    satisfying the four axioms of \cite[Definition 4.2]{biermandepaiva2000}. 
    A \textbf{categorical model of (IS4)} is an internal cartesian closed 
    category $\mathbb C$ with coproducts together with a monoidal comonad 
    $d\colon \mathbb C\to \mathbb C$ and a $d$-strong monad $t\colon 
    \mathbb C\to\mathbb C$. A \textbf{categorical order model of (IS4)} is one 
    where $\mathbb C$ is an internal preorder. \qed
\end{define}

\begin{remark}
    In any preorder with appropriate endofunctors, strength takes the form of 
    the inequality 
        \begin{equation*}
            d\wedge t \leq t(d\wedge 1).
        \end{equation*}
    and the coherence axioms of the reference are vacuous. \qed
\end{remark}

A main result of \cite{biermandepaiva2000} is that when $\E = \set$ such a 
structure affords a sound categorical interpretation of (IS4) 
\cite[Proposition 7]{biermandepaiva2000}. For this reason, we think of the 
class of structures in \cref{def:categorical-model-IS4} as providing our 
categorical semantics of (IS4).

As our categorical semantics of (IS5), we take a version of the MAO-couples of 
\cite{reyes1991} and \cite{reyeszolfaghari1991}. Again we emphasize that we are 
not adopting any particular presentation of (IS5) since there appears to be 
little if any consensus about what such a system would look like. A defining 
characteristic however should be the (B)-axiom which is modeled by the 
requirement that the monad and comonad in a categorical model of (IS4) form an 
\emph{adjoint modality}. Since MAO-couples have such an adjoint modality as a 
defining feature, we take them as our semantics. We phrase the definition as in 
\cite[\S 2.4]{reyes1991} explicitly in terms of internal preorders.

\begin{define} \label{def:MAO-couple}
    An \textbf{MAO-couple} is an internal Heyting algebra $A$ together with 
    an idempotent comonad $d\colon A\to A$ and an idempotent monad 
    $t\colon A\to A$satisfying the two axioms 
        \begin{enumerate}
            \item $t\dashv d$
            \item $t(1\wedge d) = t\wedge d$
        \end{enumerate}
    where $\wedge$ is the internal meet on $A$. \qed
\end{define}

\begin{remark}
    The statement that $t$ is left adjoint to $d$ means that two inequalities 
    hold, namely, (B.1) $1\leq dt$, what we will call the \emph{unit law}, 
    and (B.2) $td \leq 1$, the \emph{counit law}. Our order examples always 
    satisfy (B.2) but (B.1) only under the further condition 
    that the order $(X,\leq)$ is an equivalence relation. \qed
\end{remark}

\begin{lemma} \label{lemma:MAO-are-models-IS4}
    Any MAO-couple is a categorical model of (IS4).
\end{lemma}
\begin{proof}
    Any internal Heyting algebra is internally cartesian closed and cocomplete. 
    The adjoint modality forces $d$ to be internally (cartesian) monoidal. 
    Finally, the second axiom $t(1\wedge d) = t\wedge d$ implies that $t$ is 
    $d$-strong.
\end{proof}

\begin{remark}
    A categorical order model of (IS4) is \textbf{distributive} if additionally the \emph{distributive law} $td \leq dt$ holds. Distributivity is meant to encode a further relationship between the monad and comonad beyond the deducible $d \leq t$. (IS4) is not assumed to be distributive. The distributive law above is meant to interpret the further modal axiom (C) presented as 
        \begin{equation*}
            \Diamond\Box p \supset \Box\Diamond p
        \end{equation*}
    mentioned for example in discussion of the Diodorean system (S4.2) from 
    \cite{goldblatt1980}. Any symmetric order satisfies (C). Distributive laws 
    between a monad and a comonad have been studied 2-categorically 
    \cite{powerwatanabe1999}, \cite{powerwatanabe2002}, and \cite{koslowski2005}. 
    Generally, a complete statement of such distributivity requires further 
    coherence conditions. These are omitted here since they are satisfied 
    trivially by our preordered internal categories. Note that MAO-couples are 
    automatically distributive. \qed
\end{remark}

\subsection{Branching Space-time}

One of our main results is that given an equivalence relation $(X,\approx)$ in 
a topos, the power object $PX$ is an MAO-couple 
(\cref{theorem:equivalencerelationmodels(IS5)}). We shall close this section by 
introducing the setting that will yield one specific such model, namely, 
the \emph{branching space-time} of \cite{belnap1992}. The relation 
\emph{obviously undivided} (\cref{def:obviously-undivided}) will the relevant 
equivalence relation.

Let $\mscr W$ denote an ordinary poset, viewed as a category, called 
\textbf{our world}. Denote the order as `$\leq$'. Think of elements 
$e\in \mscr W$ as \emph{point events} with $e\leq e'$ meaning that from the 
perspective of $e'$, event $e$ did occur; and that from the perspective of $e$, 
the event $e'$ may or might occur. We think of such $e$ and $e'$ as being 
\emph{causally related}. Causality is formalized in the requirement that $\leq$ 
is a \emph{partial order}. In Minkowski space-time, causal precedence is a 
partial order given by certain future-directed null or future-directed timelike 
vectors. This is to be contrasted with the models of \cite{goldblatt1980} which 
are just orders and thus model \emph{chronological precedence} and are aptly 
named \emph{time-frames}.

Recall \cite[\S IX.1]{maclane1998} that a directed subset of a poset $\mscr P$ 
is a set of elements $D\subset\mscr P_0$ such that any $x,y\in D$ have a common 
element $z\in D$ with $x\leq z$ and $y\leq z$ -- briefly, if every pair of 
elements of $D$ have a common upper bound contained in $D$.

\begin{define}
    A \textbf{history} is a proper maximal directed subset $H\subset \mscr 
    W_0$. Let $H_e$ denote the set of histories of a point event $e\in\mscr W$. 
    \qed
\end{define}

By Zorn's Lemma, every directed set can be extended to a history. In 
particular, every point event belongs to some history. So, each $H_e$ is never 
empty. Histories are \emph{downward-closed}: given $e'\in H$ and $e\leq e'$, we 
must have $e\in H$ too. For $\lbrace e\rbrace \cup H$ is directed and so must 
be $H$ by maximality. So, as a result, each history is a \textit{maximal ideal} 
in our world; conversely, each such maximal ideal is a history. For this 
reason, we can write $\mspec(\mscr W)$ for the set of histories of our world. 
Although it is not strictly necessary for our theoretical development, we may 
take each history to be a \emph{Minkowski spacetime} -- that is, a 
4-dimensional real manifold with a peculiar metric attached to it. Histories in 
branching space-time organize into a presheaf on our world (cf. 
\cref{example:presheafandsheafcategoriesaretoposes}). We will call this 
presheaf the \textbf{histories presheaf}. The proof is essentially a matter of 
unpacking some previous observations.

\begin{lemma}\label{lemma:historiesformapresheaf}
    The assignment $H_{(-)}\colon \mscr W^{op}\to\set$ given by $e\mapsto H_e$ is a presheaf.
\end{lemma}
\begin{proof}
    That histories are downward-closed amounts to the fact that for causally related events $e_1\leq e_2$, it follows that $H_{e_2}\subset H_{e_1}$ holds. By reflexivity and transitivity of the causal order, this assignment defines a presheaf on our world.
\end{proof}

A running example that will figure in later considerations is the following. It 
is admittedly contrived, but is meant to illustrate the simplest example of 
non-trivial branching.

\begin{example} \label{example:simplesbranchingspacetime}
    Let $\mscr W$ denote the category with four objects and three non-identity 
    arrows as suggested by the diagram
        
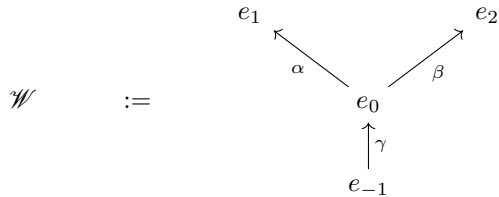
\[\begin{tikzcd}
            && {e_1} && {e_2} \\
            {\mscr W} & {:=} && {e_0} \\
            &&& {e_{-1}}
            \arrow["\alpha", from=2-4, to=1-3]
            \arrow["\beta"', from=2-4, to=1-5]
            \arrow["\gamma"', from=3-4, to=2-4]
        \end{tikzcd}\]
    In particular $\mscr W$ is a poset. It is not intended to be a \emph{real} 
    branching space-time, but just a toy model that will be used to illustrate 
    our later examples. Presumably any space-time with nontrivial branching 
    would admit at least one order-preserving morphism from $\mscr W$. Let 
    $h_1$ and $h_2$ denote the histories 
        \begin{equation*}
            h_1 = \lbrace e_{-1}\leq e_0\leq e_1 \rbrace \qquad h_2 = \lbrace e_{-1}\leq e_0\leq e_1\rbrace
        \end{equation*}
    thought of respectively as \emph{$h_1$ occurs} and \emph{$h_2$ occurs}. 
    Note that these are the only two histories. The histories presheaf 
    therefore makes the following assignments:
        \begin{align*}
            H_{e_{-1}} &= H_{e_0} = \lbrace h_1, h_2\rbrace\\
            H_{e_1} &= \lbrace h_1\rbrace \\
            H_{e_2} &= \lbrace h_2\rbrace 
        \end{align*}
    A complete picture of the image of $H$ in $\set$ is therefore 
        \[\begin{tikzcd}
            && {\lbrace h_1\rbrace} && {\lbrace h_1\rbrace} \\
            H & {=} && {\lbrace h_1,h_2\rbrace} \\
            &&& {\lbrace h_1,h_2\rbrace}
            \arrow["{-\cdot\alpha}"', hook, from=1-3, to=2-4]
            \arrow["{-\cdot\beta}", hook', from=1-5, to=2-4]
            \arrow["{-\cdot\gamma}", from=2-4, to=3-4]
        \end{tikzcd}\]
    where the hooked arrows are the singleton inclusions and $-\cdot\gamma$ is 
    in fact just the identity arrow. We shall return to this example 
    particularly in \cref{subsection:firstorderconsiderations} where it will be 
    used to study first-order modal formulas. \qed
\end{example}

\section{Internalized Alexandrov Topology}
\label{section:internalizedalexandrovtopology}

This section presents our foundational theoretical results. As the title 
suggests, these consist in internalizing the Alexandrov topology in an 
arbitrary topos. The point is that certain interior and closure operators 
on power objects can be associated to any suitably ordered object in a topos. 
Once these are constructed, we prove internalized identities that will 
eventually validate the laws of intuitionistic S4 and S5 under suitable 
interpretation.

\subsection{Interior and Closure Internally}
\label{subsection:interiorandclosureinternally}

Internal orders and their up- and downsegment arrows 
(\cref{define:downsegmentupsegmentarrows}) lead to what will be our 
semantics of the $\Box$-operator. Note that this definition does not require 
the relation in $(X,r)$ to be reflexive or transitive.

\begin{define} \label{definition:internalinterioroperator}
    The \textbf{interior operator} associated to a relation $(X,r)$ is 
    the composite morphism
        \begin{equation*}
            \interior:=P(\uparrow)\circ \downarrow\colon PX\to PX. 
        \end{equation*}
    where $P(\uparrow)\colon P^2X\to PX$ is the result of the power-object 
    functor $P\colon\mscr E^{op}\to\mscr E$ applied to the upsegment arrow 
    $\uparrow\colon X\to PX$ and $\downarrow\colon P^2X\to PX$ is the 
    downsegment arrow for the internal Heyting algebra $PX$ as in 
    \cref{define:downsegmentupsegmentarrows}.  \qed
\end{define}

\begin{remark}
    Note that in ordinary set theory, when the relation is a preorder, using 
    the inverse image equation 
    \cref{equation:settheoreticinverseimagecalculation}, this composite has the 
    effect 
        \begin{equation}\label{equation:rightadjointsuplatticehomomforalexandrovtopology}
            P(\uparrow)\!\downset (S) = \lbrace x\in X\mid \upset x\subset S\rbrace = \lbrace x\in X\mid x\leq y \text{ implies } y\in S\rbrace
        \end{equation}
    for any $S\subset X$, which is precisely the interior in the Alexandrov 
    topology on $(X,r)$ as calculated in 
    \cref{equation:interiorAlexandrovTopologyElementDescription}. One purpose 
    of this section is to justify the notation by showing that $\interior$ is 
    an internal comonad in any topos $\mscr E$. This will be done explicitly in 
    \cref{subsection:validatingcomonadaxioms}. \qed
\end{remark}

Now, we shall develop the closure operator on $PX$. This depends on the fact 
that each power-object $PX$ is an internal sup lattice as in 
\cref{definition:suplattice}. One construction of the join is given in 
\cite[\S B2.3]{johnstone2002} using the machinery of fibered limits and colimits. This is theoretically elegant, but would require a lengthy 2-categorical digression to unpack the construction in a useful way. Another approach is that of \cite[\S 5.4]{wood2004} but does not quite work with our present set-up. The most conducive and relevant explicit construction is given in \cite[\S 5.3]{johnstone2014} and is due originally to \cite{mikkelsen1976}. 

\begin{construction} \label{construction:internaljoin}
    The first point to consider is that $\Omega$ itself is an internal sup lattice. That it is an internal Heyting algebra is proved for example in 
    \cite{maclanemoerdijk1992}. The internal join is then defined as
        \begin{equation} \label{equation:joinoperatoronOmega}
            \textstyle \bigvee_\Omega := \Omega^\top\colon P\Omega \to \Omega
        \end{equation}
    and shown to make $\Omega$ internally cocomplete in the proof of 
    \cite[Lemma 5.35]{johnstone2014}. More precisely, the join operator on 
    $\Omega$ is constructed as a transpose as in
        \[\begin{tikzcd}
            {1\times \Omega^1} & \Omega \\
            {1\times P\Omega} & {\Omega\times P\Omega}
            \arrow["{\in_1}", from=1-1, to=1-2]
            \arrow["{1\times\Omega^\top}", from=2-1, to=1-1]
            \arrow["\top\times1"', from=2-1, to=2-2]
            \arrow["{\in_\Omega}"', from=2-2, to=1-2]
        \end{tikzcd}\]
    as in the definition of the contravariant exponentiation functor 
    $A^-\colon \mscr E^{op}\to \mscr E$ for any object $A$. The work of the 
    general construction of the join on $PX$ is arriving at a subsidiary 
    classifying morphism $\phi$ whose transpose composed with the join on 
    $\Omega$ transposes to the desired join on $PX$. For this, consider the 
    subobject classified by the membership arrow on $PX$: 
        \[\begin{tikzcd}
            {U_{PX}} & 1 \\
            {PX\times P^2X} & \Omega.
            \arrow["{!}", from=1-1, to=1-2]
            \arrow["{\langle c,d\rangle}"', from=1-1, to=2-1]
            \arrow["{\in_{PX}}"', from=2-1, to=2-2]
            \arrow["\top", from=1-2, to=2-2]
        \end{tikzcd}.\]
    Set-theoretically, $U_{PX}$ is the collection of pairs consisting of a 
    subset of $X$ and a collection of subsets of $X$ to which the former 
    belongs; $c$ projects that set while $d$ projects the collection. Now, form 
    the product arrow of the transposes of $c$ and of $1_X\times d$, and let 
    $\phi$ denote the classifying arrow of its image:
        \[\begin{tikzcd}
            {X\times U_{PX}} & I & 1 \\
            & {\Omega\times X\times P^2X} & \Omega
            \arrow["{!}", from=1-2, to=1-3]
            \arrow["m", tail, from=1-2, to=2-2]
            \arrow["\phi"', from=2-2, to=2-3]
            \arrow["{\langle 1_X\in_Xc,1_X\times d\rangle}"'{pos=0.2}, curve={height=12pt}, from=1-1, to=2-2]
            \arrow["e", two heads, from=1-1, to=1-2]
            \arrow["\top", from=1-3, to=2-3]
        \end{tikzcd}.\]
    Transpose $\phi$ to $\hat\phi\colon X\times P^2X\to P\Omega$ and then the 
    join on $PX$ is the unique arrow $\bigvee_{PX}$ making the diagram
        \[\begin{tikzcd}
            {X\times PX} & \Omega \\
            {X\times P^2X} & P\Omega
            \arrow["\hat\phi"', from=2-1, to=2-2]
            \arrow["{\textstyle\bigvee_\Omega}"', from=2-2, to=1-2]
            \arrow["{1\times\textstyle\bigvee_{PX}}", from=2-1, to=1-1]
            \arrow["{\in_X}", from=1-1, to=1-2]
        \end{tikzcd}\]
    commute. The point of \cite[Proposition 5.36]{johnstone2014} is showing that this indeed provides an appropriate internal left adjoint. \qed
\end{construction}

\begin{remark} \label{remark:commentsonsemanticsofjoinoperator}
    In the next subsection we will show that this join operator has a 
    reasonably intuitive semantics. We shall do this in the next subsection. 
    For now, we discuss a subtlety. Set theoretically, an element or entity $x$ 
    is a member of a union of sets in a collection $\mathcal S$ if, and only 
    if, there is a particular set in that collection $\mathcal S$ of which $x$ 
    is a member: 
        \begin{equation*}
            x\in\bigcup_{X\in \mathcal S} X \text{ if, and only if, there is } X' \in \mathcal S \text{ such that } x\in X'.
        \end{equation*}
    From the internal construction above, however, it is not at all immediately 
    clear that the existential part of the above statement is a feature of the 
    semantics of the elementary join on $PX$. As a first pass, putting together 
    the construction of the join on $PX$ and the join on $\Omega$, we have a 
    diagram 
        \[\begin{tikzcd}
            {X\times PX} & \Omega & \Omega \\
            {X\times P^2X} & P\Omega & {\Omega\times P\Omega}
            \arrow["{\in_X}", from=1-1, to=1-2]
            \arrow[Rightarrow, no head, from=1-2, to=1-3]
            \arrow["1\times\bigvee", from=2-1, to=1-1]
            \arrow["\hat\phi"', from=2-1, to=2-2]
            \arrow["{\langle\top_{P\Omega},1\rangle}"', from=2-2, to=2-3]
            \arrow["{\in_{\Omega}}"', from=2-3, to=1-3]
            \arrow["\bigvee"', from=2-2, to=1-2]
        \end{tikzcd}\]
    implying that for any element $\langle x,s\rangle\colon W\to X\times P^2X$, 
    we have an equivalence
        \begin{equation} 
            \label{equation:preliminarycharacterizationofmembershipininternaljoin}
            \top_W= x\in_{PX}\bigvee s \text{ if, and only if, } 
            \top_W= \top_W\in_{\Omega} \hat\phi(x,s).
        \end{equation}
    It turns out with a little further work that the existential element can in 
    fact be extracted from the construction. This will be done in 
    \cref{proposition:forcingsemanticsofinternaljoinoperator} 
    once we have introduced Kripke-Joyal semantics.  \qed
\end{remark}

\begin{define} \label{definition:internalclosureoperator}
    For any relation $(X,r)$, the \textbf{closure operator} associated to $r$ is 
    the composite
        \begin{eqnarray}
            \closure:=\textstyle\bigvee\exists_\downarrow\colon PX\to PX
        \end{eqnarray}
    where $\bigvee$ is the join on $PX$ and $\exists_\downarrow$ is the direct 
    image of $\downarrow\colon X\to PX$. \qed
\end{define}

\begin{remark}
    Set-theoretically, using the computation in 
    \cref{equation:settheoreticdirectimagecalculation}, this operator has the 
    effect on a subset $S\subset X$
        \begin{align}
            \textstyle \bigvee_{PX}\exists_\downarrow S &= \bigcup \lbrace T\in PX\mid T= \downset s \text{ for some } s\in S\rbrace \\
            &= \bigcup \lbrace \downset s\mid s\in S\rbrace \\
            &= \lbrace x\in X\mid x\leq s \text{ for some } s\in S\rbrace
        \end{align}
    which is precisely the Alexandrov closure of $S$ as calculated in \cref{equation:closureAlexandrovTopologyElementDescription}. \qed
\end{remark}

\subsection{Operator Semantics}
\label{subsection:operatorsemantics}

Here we extend the well-known Kripke-Joyal forcing semantics to the interior 
and closure operators of the previous section. This does require some fairly 
detailed calculations, but the upshot is that we obtain familiar-looking 
semantics of interior and closure, as well as satisfying proofs that the 
constructions validate the monad and comonad axioms given appropriate 
conditions on the underlying relation. Along the way, we will give some 
subsidiary results facilitating our calculations. Although the results for the 
closure and interior operators appear to be new, we make no claim to 
originality with regard to the required lemmas. Their importance in the 
development however warrants complete proofs.

Define the \textbf{forcing relation} `$W\Vdash\phi(a)$', read \emph{$W$ forces 
$\phi$ at $a$}, to hold just in case the generalized element $a\colon W\to X$ 
factors through the subobject classified by the proposition $\phi\colon X\to 
\Omega$ as in 
    \[\begin{tikzcd}
        & {\lbrace x\mid \phi\rbrace} & 1 \\
        W & X & \Omega
        \arrow["{!}", from=1-2, to=1-3]
        \arrow["m", from=1-2, to=2-2]
        \arrow["\phi"', from=2-2, to=2-3]
        \arrow["\top", from=1-3, to=2-3]
        \arrow["a"', from=2-1, to=2-2]
        \arrow["\exists", dashed, from=2-1, to=1-2]
    \end{tikzcd}.\]
In other words, \emph{$W$ forces $\phi$ at $a$} is understood as the relation 
    \begin{equation*}
        W\Vdash\phi(a) \text{  if, and only if,  } \phi(a) = \top_W.
    \end{equation*}
A central result of topos theory is that this relation extends to the logic of 
$\Omega$. We will only need the special cases of conjunction, implication and 
quantification, so these are quoted from standard references in the following 
result.

\begin{theo}[Forcing Semantics] \label{theorem:kripkejoyalsemantics}
    Let $a\colon W\to X$ denote any element and $\phi,\psi \colon X\rightrightarrows \Omega$ propositions on $X$. Then
        \begin{enumerate}
            \item $W\Vdash \phi(a)\wedge \psi(a)$ if, and only if, $W\Vdash \phi(a)$ and $W\Vdash \psi(a)$
            \item $W\Vdash \phi(a)\Rightarrow\psi(a)$ if, and only if, $V\Vdash \psi(ab)$ whenever $V\Vdash \phi(ab)$ for any arrow $b\colon V\to W$.
        \end{enumerate}
    If $\phi\colon X\times Y\to \Omega$ is a proposition in two unknowns, then 
        \begin{enumerate}
            \item $W\Vdash \forall y \phi(a,y)$ if, and only if, $V\Vdash \phi(ap, b)$ for all generalized elements $p\colon V\to W$ and $b\colon V\to Y$
            \item $W\Vdash \exists y \phi(a,y)$ if, and only if, $V\Vdash \phi(ap, b)$ for some epimorphism $p\colon V\twoheadrightarrow W$ and some element $b\colon V\to Y$.
        \end{enumerate}
\end{theo}
\begin{proof}
    This is proved in standard texts. For example, see the proof of 
    \cite[Theorem VI.6.1]{maclanemoerdijk1992}.
\end{proof}

The following result is probably folkloric, but its subsequent utility suggests 
that a proof is appropriate. This allows element-wise checks of internal subset 
containments and will be used in many arguments.

\begin{cor} \label{corollary:elementwisesubsetcontainment}
    For arrows $f,g\colon W\rightrightarrows PX$, we have 
        \begin{equation*}
            W\Vdash f\leq g \text{ if, and only if, } W\Vdash \forall x ((x\in_X f)\Rightarrow_\Omega (x\in_X g))
        \end{equation*}
\end{cor}
\begin{proof}
    Using the transpose isomorphism \cref{equation:isomorphismofHeytingalgebras}, 
    we have 
        \begin{equation*}
            W\Vdash f\leq g \text{ if, and only if, } X\times W\Vdash (1_X\in_X f)\leq (1_X\in_Xg)
        \end{equation*}
    in the order on $\Omega$. But since the implication operator on $\Omega$ is 
    the classifying arrow of the order, we thus have
        \begin{equation*}
            W\Vdash f\leq g \text{ if, and only if, } X\times W\Vdash ((1_X\in_X f)\Rightarrow_\Omega (1_X\in_Xg))
        \end{equation*}
    which now implies the conclusion by the implication and universal 
    quantification clauses of \cref{theorem:kripkejoyalsemantics}.
\end{proof}

As a first application, we have the Kripke-Joyal semantics of the interior 
operator (\cref{definition:internalinterioroperator}).

\begin{prop}[Forcing Semantics of Interior] \label{proposition:interiorsemantics}
    If $(X,r)$ is a relation, then 
        \begin{equation*}
            W\Vdash a\in_X \interior (s) \text{ if, and only if, } W\Vdash \forall y((a\leq y) \Rightarrow_\Omega (y\in_X s))
        \end{equation*}
    for any elements $\langle a,s\rangle \colon W\rightrightarrows X\times PX$.
\end{prop}
\begin{proof}
    By construction, $\interior\colon PX\to PX$ is the transpose of $\uparrow
    \leq 1$ since it occurs in the diagram 
        \[\begin{tikzcd}
            {X\times PX} && \Omega \\
            {X\times P^2X} & {PX\times P^2X} & \Omega \\
            {X\times PX} & {PX\times PX}
            \arrow["{\in_X}", from=1-1, to=1-3]
            \arrow["{1\times P(\uparrow)}", from=2-1, to=1-1]
            \arrow["{1\times \downarrow}", from=3-1, to=2-1]
            \arrow["{\uparrow\times 1}", from=2-1, to=2-2]
            \arrow["{\in_{PX}}", from=2-2, to=2-3]
            \arrow[Rightarrow, no head, from=2-3, to=1-3]
            \arrow["\uparrow\times1"', from=3-1, to=3-2]
            \arrow["1\times\downarrow"', from=3-2, to=2-2]
            \arrow["\leq"', curve={height=12pt}, from=3-2, to=2-3]
        \end{tikzcd}.\]
    Consequently, we have a characterization 
        \begin{equation*}
            (x\in_X \interior(s)) = (\upset x \leq_{PX} s)
        \end{equation*}
    for any $\langle x,s\rangle \colon W\to X\times PX$. Accordingly, we have 
    the equivalence
        \begin{equation*}
            W\Vdash (x\in_X \interior(s)) \text{ if, and only if, } W\Vdash (\upset x \leq_{PX} s).
        \end{equation*}
    The desired result now follows by the element-wise characterization of subset 
    containment in  \cref{corollary:elementwisesubsetcontainment} and the 
    construction of $\uparrow\colon X\to PX$.
\end{proof}

Our next task is to give analogous semantics of the closure operator. Since the 
closure is a composite of an internal join and an internal direct image, we 
shall study the semantics of each in turn. The result for direct image is 
essentially that any $a$ in $\exists_fs$ is of the form $a=fb$ for some $b$ in 
$s$, as is true in ordinary set theory.

\begin{prop} \label{lemma:forcingsemanticsofdirectimage}
    For any morphism $f\colon A\to B$, we have 
        \begin{equation*}
            W\Vdash a\in_B\exists_f(s) \text{ if, and only if, } W\Vdash \exists y ((a =_B fy)\wedge (y\in_A s))
        \end{equation*}
    for any $\langle a,s\rangle \colon W\to A\times PA$.
\end{prop}
\begin{proof}
    For the forward direction, assume that $W\Vdash a\in_B\exists_f(s)$. From the direct image construction in \cref{subsection:toposesandinternallogic}, this means equivalently that $W\Vdash \chi\langle a,s\rangle$ where $\chi$ is the classifying morphism of the image of the composite $U_X\to A\times PA\to B\times PA$ in the construction. Thus, $\langle a,s\rangle$ factors through that image via a morphism $w$ as in the diagram:
        \[\begin{tikzcd}
            V & {U_A} \\
            W & I & 1 \\
            & {B\times PA} & \Omega
            \arrow["{!}", from=2-2, to=2-3]
            \arrow["m"', tail, from=2-2, to=3-2]
            \arrow["\chi"', from=3-2, to=3-3]
            \arrow["q", dashed, from=1-1, to=1-2]
            \arrow["e"', two heads, from=1-2, to=2-2]
            \arrow["p"', dashed, two heads, from=1-1, to=2-1]
            \arrow["w", from=2-1, to=2-2]
            \arrow["{\langle a,s\rangle}"'{pos=0.2}, curve={height=12pt}, from=2-1, to=3-2]
            \arrow["\top", from=2-3, to=3-3]
        \end{tikzcd}.\]
    This diagram also depicts $e$ from the construction of $\exists_f$ and its 
    pullback $p$ along $w$ given by the dashed arrows. Note that $p$ is an 
    epimorphism since $e$ is one. Now, since by construction $me$ factors as 
    $me = (f\times 1)\langle b,a\rangle$, we have two equations, namely, 
        \begin{equation*}
            ap = fbq \text{ and } sp=aq.
        \end{equation*}
    The first of these implies that $V\Vdash ap=_Bfbq$ holds as a statement of 
    internal equality. The second, along with the computation
        \begin{equation*}
            \chi\langle a,s\rangle p = \chi meq = (bq\in_A aq) = (bq\in_A sp)
        \end{equation*}
    implies that $V\Vdash bq \in_A sp$ too. Putting the two forcing statements 
    together as a conjunction and then existentially quantifying, we get that 
    the desired forcing statement $W\Vdash \exists y ((a =_B fy)\wedge (y\in_A s))$ 
    by  \cref{theorem:kripkejoyalsemantics} since $p$ is an epimorphism. We leave 
    the converse to the reader's ingenuity since all the steps in the forward 
    direction are essentially reversible.
\end{proof}

Now, having characterized when an element is in a direct image, an examination 
of the semantics of the internal join makes good on the promise of 
\cref{remark:commentsonsemanticsofjoinoperator}. As a reminder, we would like to 
see that membership in a join of a family of subsets is witnessed by membership 
in a particular element of that family. The next result shows that this is 
indeed the case. Note that the proof bears some formal similarity to the proof 
of the previous result. Throughout we use the notation of 
\cref{construction:internaljoin}.

\begin{prop}[Forcing Semantics of Join] 
    \label{proposition:forcingsemanticsofinternaljoinoperator}
    Let $X$ be any object. The join operator $\bigvee\colon P^2X\to PX$ then satisfies
        \begin{equation*}
            W\Vdash x\in_X\bigvee s \text{ if, and only if, } W\Vdash \exists y(x\in_Xy \wedge y\in_{PX} s)
        \end{equation*}
    for any element $\langle x,s\rangle\colon W\to X\times P^2X$. In other words, $x$ is in the join if, and only if, there is a member $y$ of the collection of which $x$ is a member.
\end{prop}
\begin{proof}
    We shall prove the forward direction. Assume that $W\Vdash x\in_{PX}\bigvee 
    s$ holds. Equivalently, then, $W\Vdash \top_W\in_{\Omega} \hat\phi(x,s)$ and 
    indeed $W\Vdash \phi\langle \top_W,\langle x,s\rangle\rangle$ both hold by 
    the equivalence above in 
    \cref{equation:preliminarycharacterizationofmembershipininternaljoin} and 
    then the construction of $\hat\phi$. The last forcing statement implies 
    that $\langle \top_W,\langle x,s\rangle\rangle$ factors through $I$ via 
    some arrow, say, $w\colon W\to I$. Consider now the diagram
        \[\begin{tikzcd}
            V & {X\times U_{PX}} \\
            W & I & 1 \\
            & {\Omega\times X\times P^2X} & \Omega
            \arrow["{!}", from=2-2, to=2-3]
            \arrow["m", tail, from=2-2, to=3-2]
            \arrow["\phi"', from=3-2, to=3-3]
            \arrow["{\langle \top_W,\langle x,s\rangle\rangle}"'{pos=0.2}, curve={height=12pt}, from=2-1, to=3-2]
            \arrow["w", from=2-1, to=2-2]
            \arrow["\top", from=2-3, to=3-3]
            \arrow["e"', two heads, from=1-2, to=2-2]
            \arrow["q", dashed, from=1-1, to=1-2]
            \arrow["p"', dashed, two heads, from=1-1, to=2-1]
        \end{tikzcd}\]
    whose dashed arrows realize $V$ as the corner object of the pullback of $e$ 
    along $w$. Note that $p$ is epi. Since by construction $me$ factors $\langle 1\in c,1\times d\rangle$, we have 
        \begin{equation*}
            \langle \top_V,\langle xp, sp\rangle\rangle = \langle (1\in c)q,(1\times d)q\rangle
        \end{equation*}
    hence writing $q = \langle q_1,q_2\rangle$ and comparing arguments in the 
    above equation, we get three further equations
        \begin{equation*}
            \top_V = q_1\in_Xcq_1 \text{ and } xp= q_1 \text{ and } sp = dq_2.    
        \end{equation*}
    Consequently, putting together the first two, we have
        \begin{equation*}
            V\Vdash xp\in_Xcq_2
        \end{equation*}
    and from the commutative diagram given by the third equation 
        \[\begin{tikzcd}
            V & {U_{PX}} & 1 \\
            {U_{PX}\times W} & {PX\times P^2X} & \Omega
            \arrow["{q_2}", from=1-1, to=1-2]
            \arrow["{!}", from=1-2, to=1-3]
            \arrow["\top", from=1-3, to=2-3]
            \arrow["{\langle c,d\rangle}", from=1-2, to=2-2]
            \arrow["{\langle q_2,p\rangle}"', from=1-1, to=2-1]
            \arrow["{\langle b,s\rangle}"', from=2-1, to=2-2]
            \arrow[from=2-2, to=2-3]
        \end{tikzcd}\]
    we have finally
        \begin{equation*}
            V\Vdash cq_2\in_{PX} sp.
        \end{equation*}
    By the forcing semantics for $\wedge$ summarized in 
    \cref{theorem:kripkejoyalsemantics}, these last two forcing statements are 
    equivalent to
        \begin{equation*}
            V\Vdash (xp\in_Xcq_2 \wedge cq_2\in_{PX} sp)
        \end{equation*}
    hence to 
        \begin{equation*}
            W\Vdash \exists y(x\in_Xy \wedge y\in_{PX} s)
        \end{equation*}
    by the forcing semantics for $\exists$ in  \cref{theorem:kripkejoyalsemantics} 
    since $p$ is an epimorphism. This completes the proof of the forward 
    implication. The conscientious reader can reverse the implications to complete the proof of the converse.
\end{proof}

\begin{cor}[Forcing Semantics of Closure] 
    \label{corollary:closureforcingsemantics}
    If $(X,r)$ is any relation, then
        \begin{equation*}
            W\Vdash a\in_X\closure (s) \text{ if, and only if, } W\Vdash \exists y ((a\leq y)\wedge (y\in_Xs))
        \end{equation*}
    for any element $\langle a,s\rangle\colon W\to X\times PX$.
\end{cor}
\begin{proof}
    The left side of the equivalence is characterized as:
        \begin{align}
            W\Vdash a\in_X\closure (s) &\text{ if, and only if, } W\Vdash \exists y (a\in_Xy \wedge y\in_{PX}\exists_\downarrow (s)) \notag \\
                                       &\text{ if, and only if, } V\Vdash ap\in_X b \text{ and } V\Vdash b\in_{PX}\exists_\downarrow (sp) \notag
        \end{align}
    for an epimorphism $p\colon V\to W$ and element $b\colon V\to PX$ by the 
    previous two propositions. The right conjunct is another direct image and so 
    we have:
        \begin{align}
            V\Vdash b\in_{PX}\exists_\downarrow (sp) &\text{ if, and only if, } V\Vdash \exists y ((b=\downset y)\wedge y\in_X sp)) \notag\\
                                                     &\text{ if, and only if, } U\Vdash bq = \downset c \text{ and } U\Vdash c\in_Xspq \notag
        \end{align}
    for some epimorphism $q\colon U\to V$ and element $c\colon U\to V$. On the 
    other hand, the left conjunct can be characterized as 
        \begin{align}
            V\Vdash ap\in_Xb &\text{ if, and only if, } U\Vdash apq\in_X bq \notag \\
                             &\text{ if, and only if, } U\Vdash apq\in_X\downset c \notag \\
                             &\text{ if, and only if, } U\Vdash apq\leq c \notag
        \end{align}
    by the fact that $q$ is epi, hence right cancellable, and the computation 
    above showing that $bq$ is the principal downset on $c$. The two resulting 
    forcing statments $U\Vdash c\in_Xspq$ and $U\Vdash apq\leq c$ combine to 
    yield $W\Vdash \exists y(a\leq y\wedge y\in s)$ by 
    \cref{theorem:kripkejoyalsemantics}, yielding the right side of the desired 
    equivalence. This proves the forward direction. The reverse direction is a 
    similar argument, unwinding the definitions and using the two propositions 
    and Kripke-Joyal semantics.
\end{proof}

\begin{remark} \label{remark:informal-semantics-interior-closure}
    The equivalences characterizing membership in interiors and closures in
    \cref{proposition:interiorsemantics} and 
    \cref{corollary:closureforcingsemantics} are \emph{formal} in the sense 
    that they each incorporate internal logical operators and quantification. 
    However, in the subsequent calculations and applications, we would prefer 
    to argue in a more natural or intuitive way, ignoring where possible the 
    somewhat demanding formalism of these characterizations. Kripke-Joyal 
    semantics makes this possible. For example, owing to existential and 
    conjunction clauses of \cref{theorem:kripkejoyalsemantics}, we have that 
        \begin{equation*}
            W\Vdash a\in_X\closure (s) \text{ if, and only if, } 
                V\Vdash (ap\leq b)\text{ and } V\Vdash (b\in_Xsp))
        \end{equation*}
    for some epimorphism $p\colon V\to W$ and element $b\colon V\to X$. We allow
    ourselves to write 
        \begin{equation*}
            W\Vdash a\in_X\closure (s) \text{ if, and only if, } 
                W\Vdash (a\leq b)\text{ and } W\Vdash (b\in_Xs) \text{ for some $b$}
        \end{equation*}
    with the understanding that technically we should insert the epimorphism 
    $p$ and that it can be done in principle, but will be ignored unless it 
    matters in the computation at hand. Likewise, we have 
        \begin{equation*}
            W\Vdash a\in_X \interior (s) \text{ if, and only if, } 
            W\Vdash (a\leq b) \text{ implies } W\Vdash (b\in_X s) 
            \text{ for all $b$}
        \end{equation*}
    for interior semantics. Similarly, we will argue informally in the 
    case of internal inequalities and element-wise subset containments as 
    stated formally in \cref{corollary:elementwisesubsetcontainment}.
    \qed
\end{remark}

\subsection{Validating the (IS4) Axioms}
\label{subsection:validatingcomonadaxioms}

This section has two goals. The first is to show that if $(X,r)$ is an internal 
order, then $PX$ with the operators $\interior$ and $\closure$ is a categorical 
model of (IS4) in the sense of \cref{def:categorical-model-IS4}. Second is to 
show that if $(X,r)$ is an equivalence relation, then the model is an MAO-couple
as in \cref{def:MAO-couple}.
Wherever possible we will reason informally about interior and closure as in 
\cref{remark:informal-semantics-interior-closure}.

Throughout let $(X,r)$ denote a relation $r$ on $X$ in a topos $\mscr E$. The 
morphism $\interior:= P(\uparrow)\downarrow\colon PX\to PX$ from 
\cref{definition:internalinterioroperator} will yield our semantics of the 
modal $\Box$-operator. Likewise $\closure:=\bigvee\exists_\downarrow\colon 
PX\to PX$ of \cref{definition:internalclosureoperator} will be our semantics of 
the $\Diamond$-operator.

\begin{lemma}
    The operators $\interior,\closure\colon PX\rightrightarrows PX$ are 
    order-preserving, hence internal functors.
\end{lemma}
\begin{proof} 
    Certainly $\interior$ is one since it is a composite of functors. Likewise,
    by \cite[Proposition 5.36]{johnstone2014}, the join is order-preserving, implying that $\closure$ is a functor as well.
\end{proof}

First we shall show that $\interior$ is a comonad on $PX$. In fact it will be 
cartesian monoidal since $PX$ is preordered. Throughout we shall 
perform element-wise checks of containments using the result of \cref{corollary:elementwisesubsetcontainment}. In particular, to show the counit 
condition, namely, that $\interior \leq 1_{PX}$ holds, it suffices to prove that  
    \begin{equation*}
        W\Vdash a\in_X\interior(s) \text{ implies } W\Vdash a\in_Xs
    \end{equation*}
holds for any $a\colon W\to X$. That this indeed is the case is the content of 
the proof of the first result.

\begin{lemma} \label{lemma:reflexivityproof}
    If the relation $(X,r)$ is reflexive, then $\interior \colon PX\to PX$ 
    satisfies the counit law $\interior \leq 1_{PX}$.
\end{lemma}
\begin{proof}
    Suppose then that $W\Vdash a\in_X \interior (s)$ holds; we shall prove that 
    $W\Vdash a\in_Xs$ holds. But the hypothesis unpacks as the valid implication
        \begin{equation*}
            W\Vdash a\leq_r b \text{ implies } W\Vdash b\in_X s
        \end{equation*}
    for any $b\colon W\to X$. But by reflexivity, $W\Vdash a\leq_r a$ always 
    holds. This implies by the statement immediately above that 
    $W\Vdash a\in_X s$ as required.
\end{proof}

Now, to show the comultiplication law, we want to show that the inequality 
$\interior\leq \interior \, \interior$ holds. This amounts to showing that 
    \begin{equation} \label{equation:statementtobeprovedtransitivityproof}
        W\Vdash a\in_X \interior(s) \text{  implies  } W\Vdash a\in_X \interior \,\interior(s)
    \end{equation}
for any generalized element $\langle a,s\rangle\colon W\to X\times PX$. We can 
make a couple of reductions before embarking upon the proof. The consequent of 
the desired statement 
\cref{equation:statementtobeprovedtransitivityproof} is equivalent to
    \begin{equation} \label{equation:statementtobeprovedtransitivityproofredux}
        W\Vdash a\leq b \text{  implies  } W\Vdash b\in_X \interior(s)
    \end{equation}
for any element $b\colon W\to X$. Thus, we will assume the antecedents of 
\cref{equation:statementtobeprovedtransitivityproof} and 
\cref{equation:statementtobeprovedtransitivityproofredux} and show the above 
consequent, namely, $W\Vdash b\in_X \interior(s)$.

\begin{prop} \label{proposition:AnOrderedObjectModelsComonadAxioms}
    If the relation $(X,r)$ is an order, then $\interior$ satisfies the 
    comultiplication law $\interior\leq \interior \,\interior$. In particular 
    $\interior$ is a cartesian monoidal comonad on $PX$.
\end{prop}
\begin{proof}
    As discussed above, assume that $W\Vdash a\in_X \interior(s)$ and 
    $W\Vdash a\leq b$ both hold for appropriate generalized elements. The first 
    assumption yields the equivalence: 
        \begin{equation*}
            W\Vdash a\in_X \interior(s) \text{ if, and only if, } W\Vdash a\leq c \text{ implies } W\Vdash c \in_X s
        \end{equation*}
    for any element $c\colon W\to X$. The desired statement is formally analogous:
        \begin{equation*}
            W\Vdash b\in_X \interior(s) \text{ if, and only if, } W\Vdash b\leq d \text{ implies } W\Vdash d \in_X s
        \end{equation*}
    for any $d\colon W\to X$. We can prove the right side of the equivalence 
    immediately above. For if $W\Vdash b\leq d$ for such $d$, then since 
    $W\Vdash a\leq b$ by transitivity we conclude that $W\Vdash a\leq d$ too. 
    Therefore, since $a$ is an interior point, $W\Vdash d\in_X s$ holds. For 
    product-preservation, it suffices to show that $\downarrow\colon PX\to P^2X$ 
    is product-preserving since $P(\uparrow)$ is a right adjoint. But 
    $\downarrow\colon PX\to P^2X$ is also a right adjoint since $PX$ is an 
    internal sup-lattice. The coherence axioms for a monoidal comonad are satisfied automatically since $PX$ is an internal preorder. 
\end{proof}

Now, we shall show that closure is a monad. First we shall check the unit axiom, 
namely, that $1\leq \closure$ holds. By \cref{corollary:elementwisesubsetcontainment}, 
this is equivalent to showing the implication
    \begin{equation*}
        W\Vdash a\in_X s \text{ implies } W\Vdash a\in_X\closure(s)
    \end{equation*}
for any element $\langle a,s\rangle \colon W\to X\times PX$.

\begin{lemma}
    If $(X,r)$ is a reflexive relation on $X$, then $1\leq \closure$ holds.
\end{lemma}
\begin{proof}
    Let $\langle a,s\rangle\colon W\to X\times PX$ denote any element. 
    Suppose that $W\Vdash a\in_X s$. By reflexivity $W\Vdash a\leq a$ 
    certainly holds, which implies that $W\Vdash a\in_X\closure(s)$ by closure 
    semantics.
\end{proof}

Transitivity implies the multiplication law. It suffices to prove the implication 
    \begin{equation*}
        W\Vdash a\in_X \closure\,\closure(s) \text{ implies } 
        W\Vdash a\in_X\closure(s)
    \end{equation*}
for any element $\langle a,s\rangle\colon W\to X\times PX$.

\begin{lemma} \label{lemma:closure-is-a-monad}
    If $(X,r)$ is transitive, then $\closure\,\closure\leq \closure$ holds. 
    Consequently, $\closure$ is a monad on $PX$ if $X$ is an order.
\end{lemma}
\begin{proof}
    Assume the antecedent of the implication in the above display. This yields 
    the equivalence
        \begin{equation*}
            W\Vdash a\in_X \closure\,\closure(s) 
                \text{ if, and only if, } W\Vdash a\leq b \text{ and } 
                    W\Vdash b\in_X\closure(s)    
        \end{equation*}
    and of course the rightmost conjunct is equivalent to 
    the conjunction of $W\Vdash b\leq c$ and $W\Vdash c\in s$. But then by 
    transitivity, $W\Vdash a\leq c$ holds, meaning we can conclude that 
    $W\Vdash (a\in_X\closure (s))$ holds too, proving the implication.
\end{proof}

We now have our first main result.

\begin{theo} \label{theorem:ordermodels(IS4)}
    If $(X,r)$ is an order on $X$, then $PX$ 
    with $\interior,\closure\colon PX \rightrightarrows PX$ is a categorical 
    order model of (IS4) as in \cref{def:categorical-model-IS4}.
\end{theo}
\begin{proof}
    Any power object $PX$ is an internal Heyting algebra, hence in particular 
    internally cartesian closed and cocomplete. By 
    \cref{proposition:AnOrderedObjectModelsComonadAxioms} and 
    \cref{lemma:closure-is-a-monad}, it therefore only remains only to see that 
    $\closure$ is $\interior$-strong. By 
    \cref{corollary:elementwisesubsetcontainment}, it suffices to prove that 
        \begin{equation*}
            W\Vdash (a\in_X(\interior(s)\wedge \closure(u))) \text{ implies } 
                W\Vdash a \in_X\closure (\interior(s)\wedge u)
        \end{equation*}
    for any $a\colon W\to X$ and $s,u\colon W\rightrightarrows PX$. But this 
    follows by transtivity. Supposing the antecendent, we need to produce a $b$ 
    such that 
        \begin{equation*}
            W\Vdash a\leq b \text{ and } W\Vdash b\in_X(\interior(s)\wedge u).
        \end{equation*}
    But such a $b$ is given by the antecedent hypothesis that 
    $W\Vdash a\in_X\closure(u)$. For from this assumption $W\Vdash b\in u$ 
    holds for some $b$. It needs only to be seen that $W\Vdash b\in 
    \interior(s)$. So suppose that $c$ satisfies $W\Vdash b\leq c$. But then by 
    transitivity $W\Vdash a\leq c$, meaning that $W\vdash c\in_Xs$ since 
    $W\Vdash a\in_X \interior(s)$ by assumption.
\end{proof}

Our final goal is to show that, supposing symmetry, interior and closure make 
$PX$ into an MAO-couple (\cref{def:MAO-couple}). By 
\cref{theorem:ordermodels(IS4)},this requires showing only that 
    \begin{enumerate}
        \item $\interior$ and $\closure$ are idempotent;
        \item $\interior$ and $\closure$ form an adjoint modality with 
                $\closure \dashv \interior$.
    \end{enumerate}
Idempotence is immediate as a formal consequence of the internal preorder on 
$PX$. So, we shall focus on adjointness. Internally, this is the statement that 
    \begin{equation} \label{equation:identityforsymmetry}
        (1\leq_{PX} \interior\,\closure) \text{ and } 
        (\closure\,\interior \leq_{PX} 1)
    \end{equation}
are both both valid, that is, equal to $\top_{PX}$. That this is the case can 
be shown purely from the calculus of internal adjoints. First we need a 
preliminary lemma, explaining the connection between symmetry and the direct 
image of the downset operator.

\begin{lemma} \label{lemma:symmetryimpliesdownsetandupsetareequal}
    If $(X,r)$ is a symmetric order, then the upset and downset morphisms are 
    equal: $\uparrow = \downarrow$.
\end{lemma}
\begin{proof}
    The upper- and lower-set morphisms are constructed as transposes as in
        \[\begin{tikzcd}
            {X\times PX} & \Omega && {X\times PX} & \Omega \\
            {X\times X} &&& {X\times X} & {X\times X}
            \arrow["{\in_X}", from=1-1, to=1-2]
            \arrow["{\in_X}", from=1-4, to=1-5]
            \arrow["1\times\uparrow", from=2-4, to=1-4]
            \arrow["\cong"', from=2-4, to=2-5]
            \arrow["{\leq_r}"', from=2-5, to=1-5]
            \arrow["{1\times \downarrow}", from=2-1, to=1-1]
            \arrow["{\leq_r}"', curve={height=12pt}, from=2-1, to=1-2]
        \end{tikzcd}\]
    where the isomorphism on the right exchanges the factors in the product of $X$ with itself. But symmetry means that there is a morphism $R\to R$ such that 
        \[\begin{tikzcd}
            R & R \\
            {X\times X} & {X\times X}
            \arrow[dashed, from=1-1, to=1-2]
            \arrow["r", tail, from=1-2, to=2-2]
            \arrow["r"', tail, from=1-1, to=2-1]
            \arrow["\cong"', from=2-1, to=2-2]
        \end{tikzcd}\]
    commutes. But the square is a pullback, which implies that the two arrows $X\times X\to \Omega$ in the first display are equal by uniqueness of classifying arrows, meaning that their transposes are also equal.
\end{proof}

\begin{remark}
    The previous lemma justifies the following notation. When $(X,r)$ is an 
    equivalence relation, the classifying arrow of $r\colon R\to X\times X$ will 
    be denoted by $\approx\colon X\times X\to\Omega$. The transpose of $\approx$ 
    will be denoted by $[-]\colon X\to PX$, thought of as taking $x\in X$ to its 
    equivalence class $[x] = \lbrace y\in X\mid x\approx y\rbrace$. \qed
\end{remark}

\begin{theo} \label{theorem:equivalencerelationmodels(IS5)}
    If $(X,r)$ is an internal equivalence relation, then the closure operator 
    $\closure\colon PX\to PX$ and interior operator $\interior\colon PX\to PX$ 
    are internally adjoint: $\closure\dashv\interior$. Consequently, interior 
    and closure make the power object $PX$ into an MAO-couple 
    (\cref{def:MAO-couple}).
\end{theo}
\begin{proof}
    By \cref{lemma:symmetryimpliesdownsetandupsetareequal}, we shall write 
    $[-]$ for the downset and upset morphisms since they are equal. The two 
    identities that need to be shown are the unit and counit laws, namely, that 
    both inequalities
        \begin{equation*}
            1\leq\interior\,\closure \text{ and } \closure\,\interior \leq 1
        \end{equation*}
    are valid, that is, equal to $\top_{PX}$. Each verification is straighforward 
    by the internal calculus of adjoints. On the one hand, for the conunit 
    equation, we have that 
        \begin{align}
            \closure\,\interior \leq 1 &\text{ if, and only if, } \bigvee\exists_{[-]}\interior \leq 1 \notag \qquad &\text{(def. $\closure$ in \cref{definition:internalclosureoperator})}\\
                                       &\text{ if, and only if, } \exists_{[-]}\interior\leq [-] \notag \qquad &\text{($\bigvee \dashv [-]$)} \\
                                       &\text{ if, and only if, } \interior \leq P([-])[-] \notag \qquad &\text{($\exists_{[-]}\dashv P([-])$)}
        \end{align}
    Since the last inequality is clearly valid by the definition of $\interior$ 
    in \cref{definition:internalinterioroperator}, this proves the counit 
    inequality. Notice that symmetry is used here in the last step since 
    $\interior$ has both up- and down-set operators in its definition, but these 
    are equal by the lemma. The other required inequality has a similar proof.
\end{proof}

\subsection{The Obviously Undivided Relation}

With the main theorem proved, we can now return to the application of the 
histories presheaf from the end of \cref{subsection:toposmodelsgenerally}. Again let $\mscr W$ denote the poset \emph{our world} and let $H\colon \mscr W^{op}\to\set$ denote the histories presheaf sending a point event $e$ to the set $H_e$ of histories containing it. First we shall define the relation \emph{obviously undivided} at $e$ leading to our model.

\begin{define} \label{def:obviously-undivided}
    Two histories $h_1$ and $h_2$ containing a point event $e$ are 
    \textbf{obviously undivided at} $e$ if both $h_1$ and $h_2$ contain some 
    further point event $e'$ strictly later than $e$, in that $e\leq e'$ and 
    $e\neq e'$, if there are any such point events. Otherwise $h_1$ and $h_2$, 
    provided they each contain $e$, \textbf{apparently diverge} at $e$. Write 
    $h_1\approx_e h_2$ when $h_1$ and $h_2$ are obviously undivided at the 
    common point event $e$.\qed
\end{define}

\begin{example} \label{example:simplesbranchingspacetimeredux}
    Return to the situation of \cref{example:simplesbranchingspacetime} 
    where $\mscr W$ consists of the four point events and three nontrivial arrows 
    as in 
        
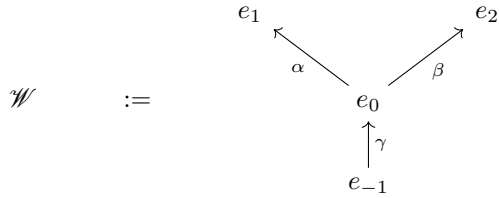
\[\begin{tikzcd}
            && {e_1} && {e_2} \\
            {\mscr W} & {:=} && {e_0} \\
            &&& {e_{-1}}
            \arrow["\alpha", from=2-4, to=1-3]
            \arrow["\beta"', from=2-4, to=1-5]
            \arrow["\gamma"', from=3-4, to=2-4]
        \end{tikzcd}\]
    This example, however simple, captures some of the fundamental ideas of 
    branching and the significance of \textbf{choice points}, namely, those 
    point events $e$ which are maximal in the intersection of two histories 
    \cite[Fact 25, p.407]{belnap1992}. In this example, $e_0$ is the unique 
    choice point. Since $h_1$ and $h_2$ contain $e_{-1}$ and $e_0$ is properly 
    later than $e_{-1}$, we have that $h_1\approx_{e_{-1}}h_2$, i.e., $h_1$ and 
    $h_2$ are obviously undivided at $e_{-1}$. Roughly, events around any choice 
    point in a \emph{real} branching space-time would intuitively look something 
    like $\mscr W$. So, we think of this example as the smallest model of 
    non-trivial branching and obvious undividedness.
    \qed
\end{example}

Notice that if $e$ is a maximal point event in our world histories $h_1$ and 
$h_2$ containing such $e$ are obviously undivided at $e$ by vacuity. Thus, 
$\approx_e$ is reflexive. It is also trivially symmetric. Two further assumptions 
on our world make the relation $\approx_e$ into a particularly nice one. First 
is the assumption of \textbf{existence of infima}: each nonempty chain that is 
bounded below has an infimum in each history of which it is a subset. Secondly, 
there is the assumption of \textbf{density}, namely, that given $e\leq e'$ there 
is a point event properly between $e$ and $e'$. Both for example are true of 
Minkowski space-time.

\begin{lemma} \label{lemma:obviouslyundividedisanequivalencerelation}
    Assume that our world satisfies the infima and density assumptions summarized above. For each point event $e$ of our world, $\approx_e$ is then an equivalence relation on $H_e$.
\end{lemma}
\begin{proof}
    As remarked above, evidently $\approx_e$ is reflexive and symmetric. It has 
    only to be seen that it is transitive. The proof of Fact 46 in 
    \cite[\S A.2.]{belnap1992} shows this given the assumptions of density and 
    existence of infima.
\end{proof}

The final task in this section is that of showing that the pointwise-defined 
\emph{obviously undivided} relation induces an equivalence relation 
(\cref{definition:orders}) on the histories presheaf of 
\cref{lemma:historiesformapresheaf}. This will follow by a lemma showing that 
equivalence relations in a presheaf topos can be defined pointwise.

\begin{lemma}\label{lemma:pointwisecharacterofequivalencerelations}
    A relation $R\to P\times P$ in a presheaf category $[\mscr C^{op},\set]$ is 
    an equivalence relation if, and only if, $Rc\to Pc\times Pc$ is one in 
    $\set$ for each $c\in\mscr C$.
\end{lemma}
\begin{proof}
    The forward direction is a triviality. For the converse, suppose that each $Rc\to Pc\times Pc$ is an equivalence relation. The proofs consist merely in showing that any induced morphisms factoring through $Rc$ comprise a natural family. For symmetry this is immediate. The arguments for reflexivity and transitivity are identical in strategy. For example, instantiating transitivity, for each object $c$, the corner object $Tc$ factors through $Rc$ as in 
        \[\begin{tikzcd}
            Tc & Rc \\
            & {Pc\times Pc}
            \arrow["{\psi_c}", dashed, from=1-1, to=1-2]
            \arrow["{\langle ap,bq\rangle_c}"', curve={height=6pt}, from=1-1, to=2-2]
            \arrow["{\langle a,b\rangle_c}", from=1-2, to=2-2]
        \end{tikzcd}\]
    These morphisms $\psi_c$ comprise a natural transformation. For each $\langle a,b\rangle_c$ is monic, i.e. injective in $\set$. So, for a morphism $f\colon c\to d$, we calculate that
        \begin{equation*}
            \langle a,b\rangle_c R(f)\psi_d = (Pf\times Pf)\langle ap,bq\rangle = \langle a,b\rangle_c\psi_c Tf
        \end{equation*}
    which implies the naturality equation $R(f)\psi_d = \psi_cTf$ by cancelling $\langle a,b\rangle_c$. Again the same strategy works to show reflexivity.
\end{proof}

\begin{theo} \label{theorem:historiespresheafmodelsIS5}
    The obviously undivided relation induces an equivalence relation on the 
    histories presheaf. Therefore, the histories presheaf $H$ canonically 
    induces an MAO-couple in the presheaf topos on our world.
\end{theo}
\begin{proof}
    It suffices by  \cref{lemma:pointwisecharacterofequivalencerelations} to do 
    a pointwise check. But this is confirmed by 
    \cref{lemma:obviouslyundividedisanequivalencerelation} which showed that 
    each $\approx_e$ is an equivalence relation on the set $H_e$ of histories 
    containing $e$. The model statement follows now just by noting 
    \cref{theorem:equivalencerelationmodels(IS5)}.
\end{proof}

\section{Applications, Interpretations and Computations}
\label{section:branchingspacetime}

This section develops illustrations, elaborations, and examples of the previous 
results. In \cref{subsection:presheafcategorycomputations}, we carry out a 
number of computations in presheaf categories that illustrate the abstractions 
more concretely. These are applied in \cref{subsection:indistinguishability} to 
discuss the so-called \emph{indistinguishability interpretation} of equivalence 
relations providing the natural models for \emph{epistemic logic}. We shall see 
that forcing statements arising from the obviously undivided relation have a 
natural interpretation in this light. Presheaf computations are also applied in 
\cref{subsection:firstorderconsiderations} in the investigation of several 
first-order modal formulas, including the Barcan formulas, existence 
predicates, and then the validity of the actualist thesis 
\cref{equation:actualist-thesis}.

\subsection{Presheaf Category Computations}
\label{subsection:presheafcategorycomputations}

This subsection is devoted to unwinding our general constructions for presheaf 
toposes $\mscr E=[\mscr C^{op},\set]$. The reference for generalities of such 
toposes is \cite{maclanemoerdijk1992}. We adopt the convention that the action 
of the transition function $X(f)\colon XC\to XD$ for an arrow $f\colon D\to C$ 
and element $a\in XC$ will be denoted by $a\cdot f := X(f)(a)$. 

Our goal is to understand the interior operator in this special case. It is 
tempting to jump to the conclusion that the forcing semantics of necessity 
statements $\Box\phi\colon X\to \Omega$ is given by 
    \begin{equation*}
        C\Vdash a\in_X \Box\phi \text{ if, and only if, } C\Vdash a 
        \leq b \text{ implies } C\Vdash \phi_C(b).
    \end{equation*}
This looks plausible given \cref{proposition:interiorsemantics} but needs to be 
proved. This will be done in detail over the course of the subsection. The resulting proposition will be used throughout \cref{subsection:firstorderconsiderations}.

First recall \cite[\S VI.7]{maclanemoerdijk1992} that, since every presheaf is 
a colimit of representables and that transformations $\alpha\colon \mathbf y 
C\to X$ are essentially elements $a\in XC$, the forcing relation takes the form 
    \begin{equation} \label{equation:presheafforcingrelation}
        C\Vdash \phi(a) \text{ if, and only if, } \phi_C(a) = \top_X \text{ if, and only if, } a\in\lbrace x\mid\phi\rbrace(C)
    \end{equation}
where $\lbrace x\mid\phi\rbrace$ denotes the subobject classified by 
$\phi\colon X\to \Omega$. In the following proposition, the morphism 
$\Box\phi\colon 1\to PX$ is the composite of $\Box:=\interior\colon PX\to PX$ 
and the transpose $\hat\phi\colon 1\to PX$ of $\phi\colon X\to \Omega$ in the 
presheaf topos $\mscr E = [\mscr C^{op},\set]$. Likewise a similar pattern 
yields the morphism $\Diamond\phi$. Forcing statements are most distinctive for 
$\Omega$-valued morphisms that are not necessarily closed formulas, so we 
consider the transposes $(1_X\in_C \Box\phi)\colon X\to \Omega$ and $(1_X\in_C 
\Diamond\phi)\colon X\to \Omega$.

\begin{prop} \label{proposition:presheafforcingsemanticsinteriorandclosure}
    For any proposition $\phi\colon X\to \Omega$ where $(X,r)$ is a relation,
        \begin{enumerate}
            \item $C\Vdash (a\in_C \Box\phi)$ if, and only if, for all $b\in XC$ with $a\leq_C b$ it follows that $C\Vdash \phi_C(b)$
            \item $C\Vdash (a\in_C \Diamond\phi)$ if, and only if, there is $b\in XC$ with $a\leq_C b$ with $C\Vdash \phi_C(b)$.
        \end{enumerate}
\end{prop}
\begin{proof}
    These statements result from instantiating 
    \cref{proposition:interiorsemantics} and 
    \cref{corollary:closureforcingsemantics} using the observation above that 
    elements $\mathbf yC\to X$ are in bijection with elements of $XC$ and that 
    the $C$-component of $b\in_X\hat\phi$ is equal to $\phi_C(b)$ by 
    transposition. A detailed examination of the necessity operator amounting 
    to a direct proof of the first clause of the proposition appears in 
    \cref{example:necessityoperatorinpresheaftoposes}.
\end{proof}

The reader comfortable with the reference, or at least the constructions 
therein, can skip most of the following generalities and look at 
\cref{example:classifyingmorphismforordersinpresheaftopos}, 
\cref{example:completedescriptionofupsetanddownsetarrowsinpresheaftoposes} and 
\cref{example:necessityoperatorinpresheaftoposes} which are the main novelties 
of the subsection. These are a complete description of the upset and downset 
arrows $\uparrow,\downarrow\colon X\rightrightarrows PX$ in any presheaf topos, 
as well as a more in-depth analysis of the interior operator and its semantics. 

\begin{construction}[Subobject Classifier]
    Let $\Omega\colon \mscr C^{op}\to\set$ denote the presheaf given by 
        \begin{equation*}
            \Omega(C) = \lbrace S\mid S\text{ is a sieve on } C\rbrace
        \end{equation*}
    on objects. On a given arrow $f\colon C\to D$, the associated function $\Omega(f)\colon \Omega(D)\to \Omega(C)$ makes the assignment 
        \begin{equation*}
            \Omega(f)(T) = f^*T = \lbrace g \mid fg\in T\rbrace
        \end{equation*}
    for any sieve $T$ on $D$. It is straighforward to check that this is a 
    sieve and that $\Omega$ so defined is a presheaf. Given a monomorphism 
    $S\to X$ in the presheaf topos, the associated classifying morphism 
    $\chi\colon X\to \Omega$ has components $\chi_C\colon X(C)\to \Omega(C)$ 
    given by 
        \begin{equation} \label{equation:classifyingmorphisminpresheaftoposes}
            \chi_C(a) = \lbrace f\colon D\to C\mid a\cdot f\in S(D)\rbrace.
        \end{equation}
    That is, $\chi_C(a)$ is the sieve of morphisms $f$ into $C$ whose action 
    under $Xf$ ends up in the appropriate component of the subobject $S$. The 
    top element of $\Omega$ is $\top\colon 1\to \Omega$ picking out the 
    \emph{total sieve} of all arrows on each object $C$. It is now 
    straighforward to verify that for each subobject of $X$, the arrow $\chi$ 
    is the unique morphism into $\Omega$ realizing the subobject as a pullback 
    of $\top$ along $\chi$. \qed
\end{construction}

\begin{example}[Classifying Morphisms for Orders] 
    \label{example:classifyingmorphismforordersinpresheaftopos}
    Let $(X,r)$ denote a relation $r\colon R\to X\times X$ on a presheaf $X$. 
    This is point-wise a relation $r_C\colon RC\to XC\times XC$ in $\set$ for 
    each object $C$. The classifying morphism $\leq_r\colon X\times X\to 
    \Omega$ fits into a pullback diagram 
        \[\begin{tikzcd}
            R & 1 \\
            {X\times X} & \Omega
            \arrow["{!}", from=1-1, to=1-2]
            \arrow["r"', tail, from=1-1, to=2-1]
            \arrow["{\leq_r}"', from=2-1, to=2-2]
            \arrow["\top", tail, from=1-2, to=2-2]
        \end{tikzcd}\]
    and by the description in 
    \cref{equation:classifyingmorphisminpresheaftoposes}, for any object $C$, 
    we have
        \begin{equation} \label{equation:actionofcomponentrelationonelements}
            a\leq_Cb = \lbrace f\colon D\to C \mid (a\cdot f\leq_Db\cdot f)\rbrace.
        \end{equation}
    Note that here we have written $\leq$ for $\leq_r$ and preferred the infix 
    notation since it is easier to parse. So, in other words, the image of 
    $(a,b)$ under the classifying arrow is the sieve of arrows $f\colon D\to C$ 
    under whose action $a$ and $b$ are related under $r$ in the sense that 
    $a\cdot f\leq b\cdot f$ in $X(D)$. We shall be interested in this 
    construction when $(X,r)$ is at least an order, but especially when it is 
    an equivalence relation.\qed
\end{example}

\begin{example}[Kernel of a Morphism]
    Let $\tau\colon X\to Y$ denote any morphism in a topos $\mscr E$. The 
    \textbf{kernel} of $\tau$ is the pullback of $\tau$ along itself 
        \[\begin{tikzcd}
            {X\times_YX} & X \\
            X & Y
            \arrow["\rho", from=1-1, to=1-2]
            \arrow["\lambda"', from=1-1, to=2-1]
            \arrow["\tau"', from=2-1, to=2-2]
            \arrow["\tau", from=1-2, to=2-2]
            \arrow["\lrcorner"{anchor=center, pos=0.125}, draw=none, from=1-1, to=2-2]
        \end{tikzcd}\]
    and the projections always yield an equivalence relation 
    $\langle\lambda,\rho\rangle\colon X\times_YX \to X\times X$. In the case 
    when $\mscr E = [\mscr C^{op},\set]$, since limits are computed pointwise 
    in $\set$, for each $C\in\mscr C$, we have a pullback square 
        \[\begin{tikzcd}
            {XC\times_{YC}XC} & XC \\
            XC & YC
            \arrow["{\rho_C}", from=1-1, to=1-2]
            \arrow["{\lambda_C}"', from=1-1, to=2-1]
            \arrow["{\tau_C}"', from=2-1, to=2-2]
            \arrow["{\tau_C}", from=1-2, to=2-2]
            \arrow["\lrcorner"{anchor=center, pos=0.125}, draw=none, from=1-1, to=2-2]
        \end{tikzcd}\]
    whose corner object consists of those pairs $(a,b)$ satisfying 
    $\tau_Ca = \tau_Cb$. Accordingly, the classifying morphism 
    $\ker(\tau)\colon X\times X\to \Omega$ has components $\ker(\tau)_C$ whose 
    action is 
        \begin{equation*}
            \ker(\tau)_C(a,b) = \lbrace f\colon D\to C\mid \tau_D(a\cdot f) = \tau_D(b\cdot f)\rbrace
        \end{equation*}
    that is, whose action produces the sieve of arrows into $C$ whose action on 
    $a$ and $b$ are equalized by the appropriate component of $\tau$. \qed
\end{example}

\begin{example}[Identity Relation]
    In any topos, the classifier of the diagonal morphism on an object $X$ 
    serves as an identity predicate $=_X\colon X\times X\to\Omega$. 
    When $X$ is a presheaf, each component of the identity predicate 
    $(=_X)_C\colon XC\times XC\to\Omega C$ has the action 
        \begin{equation*}
            a (=_X)_C b = \lbrace f\colon D\to C\mid a\cdot f = b\cdot f\rbrace\qquad a,b\in XC.
        \end{equation*}
    The equation $a\cdot f = b\cdot f$ is set-theoretic equality in $X(D)$ 
    since each component of $\Delta\colon X\to X\times X$ is the actual 
    set-theoretic diagonal. So, colloquially, $a (=_X)_C b$ is the sieve of 
    those arrows into $C$ under which the action on $a$ and $b$ are literally 
    equal in $XD$. \qed
\end{example}

The next task is to look carefully at transposes and in particular at those of 
the classifying morphisms  of orders as above. This is so we can understand the 
upset and downset arrows from \cref{define:downsegmentupsegmentarrows} in this 
context. Generally, exponential objects 
$Q^P$ are given by a homset formula, namely, 
    \begin{equation} \label{equation:power-objectcomponentinpresheaftopos}
        Q^P(C) = [\mscr C^{op},\set](\mathbf yC\times P,Q)
    \end{equation}
that is, as the set of all transformations $\mathbf yC\times P\to Q$. For a 
morphism $f\colon D\to C$, the action of $Q^P(f)$ is to take a transformation 
$\theta\colon \mathbf yC\times P\to Q$ to one 
$Q^P(f)(\theta)\colon \mathbf y D\times P\to Q$ defined by
    \begin{equation} \label{equation:transitionfunctionpower-objectspresheaftopos}
        \mscr C(B,D)\times PB\to QB \qquad (g,b)\mapsto \theta_B(fg,b)
    \end{equation}
where $g\colon B\to D$ any arrow and $b\in PB$ is any element. Since power objects 
can be viewed as special exponentials, namely, $PX = \Omega^X$, we can use this 
construction to do our calculations.\qed

\begin{construction}[Power Objects]
    For any presheaf $X\colon \mscr C^{op}\to \set$, the associated power object is the exponential $PX=\Omega^X$ whose value at any object $C$ is
        \begin{equation*}
            PX(C) = \Omega^X(C) = [\mscr C^{op},\set](\mathbf yC\times X,\Omega).
        \end{equation*}
    That is, $PX(C)$ is the set of transformations $\mathbf yC\times X\to \Omega$. Any such $\tau\colon \mathbf yC\times X\to \Omega$ has components 
        \begin{equation*}
            \tau_D\colon \mscr C(D,C)\times X(D)\to \Omega(D) \qquad\qquad \tau_D(f,a)\in \Omega(D).
        \end{equation*}
    picking out $\tau_D(f,a)$, a sieve on $D$, for each such pair $(f,a)$ consisting of an arrow $f\colon D\to C$ and an element $a\in X(D)$. Given a transformation $\tau\colon X\times Y\to \Omega$, the transpose is the unique transformation $\hat g\colon Y\to PX$ making 
        \[\begin{tikzcd}
            {X\times PX} & \Omega \\
            {X\times Y}
            \arrow["{\in_X}", from=1-1, to=1-2]
            \arrow["{1\times\hat g}", from=2-1, to=1-1]
            \arrow["g"', curve={height=6pt}, from=2-1, to=1-2]
        \end{tikzcd}\]
    commute and for each object $C$ the component $\hat g_C\colon Y(C)\to PX(C)$ has the action of sending $b\in Y(C)$ to a transformation $\mathbf yC\times X\to \Omega$ given by
        \begin{equation*}
            \hat g_C(b)_D\colon \mscr C(D,C)\times X(D) \to \Omega D \qquad \hat g_C(b)_D(f,a) = g(a,Yf(b))
        \end{equation*}
    for any object $D$ of $\mscr C$. In the special case of a transformation 
    $\phi\colon X\to \Omega$, we can compute the transpose $\hat\phi\colon 1\to 
    PX$ from this general formula. The transpose is the unique such morphism 
    making the diagram 
        \[\begin{tikzcd}
            {X\times PX} & \Omega \\
            {X\times 1\cong X}
            \arrow["{\in_X}", from=1-1, to=1-2]
            \arrow["{1\times\hat \phi}", from=2-1, to=1-1]
            \arrow["\phi"', curve={height=12pt}, from=2-1, to=1-2]
        \end{tikzcd}\]
    commute. For any object $C$, the set function $\hat\phi_C\colon 1\to PX(C)$ 
    picks out a distinguished transformation $\hat\phi_C(\ast)\colon\mathbf 
    yC\times X\to \Omega$ whose components are described by
        \begin{equation} \label{equation:actionofphihatinpresheaftopos}
            \hat\phi_C(\ast)_D\colon \mscr C(D,C)\times X(D) \to \Omega(D) \qquad \hat\phi_C(\ast)_D(f,a) = \phi_D(a).  
        \end{equation}
    Note that the formula on the right ignores the arrow $f$ and evaluates the 
    appropriate component of $\phi$ at the element $a\in X(D)$. \qed
\end{construction}

\begin{example}[Upset and Downset Arrows] \label{example:completedescriptionofupsetanddownsetarrowsinpresheaftoposes}
    Continue with the elaboration begun in 
    \cref{example:classifyingmorphismforordersinpresheaftopos}. Let $(X,r)$ 
    denote any relation in $[\mscr C^{op},\set]$. The downset arrow 
    $\downarrow\colon X\to PX$ has components $\downarrow_C\colon XC\to PX(C)$ 
    sending each $b\in XC$ to a transformation $\mathbf yC\times X\Rightarrow 
    \Omega$ whose components over $D\in \mscr C$ are described by the action 
        \begin{equation} \label{equation:downsetactiononcomponents}
            \downset_C(b)_D(f,a) = a \leq_D b\cdot f \qquad (f,a)\in \mscr C(D,C)\times XD.
        \end{equation}
    So, in other words, using the description of 
    \cref{equation:actionofcomponentrelationonelements}, we have that 
    $\downset_C(b)_D(f,a)$ is the sieve of those arrows $g\colon B\to D$ for 
    which $a\cdot g \leq_B b\cdot fg$ holds. Likewise, the upset arrow 
    $\uparrow\colon X\to PX$ has the action
        \begin{equation} \label{equation:upsetactiononcomponents}
            \upset_C(b)_D(f,a) = b\cdot f \leq_D a \qquad (f,a)\in \mscr C(D,C)\times XD.
        \end{equation}
    So, specifically, this action amounts to the sieve of those $g\colon B
    \to D$ for which $b\cdot fg \leq_B a\cdot g$ holds, effectively just 
    reversing that of $\downarrow$. We will later apply the downset arrow 
    construction to the case of $PX$ with its order induced from $\Omega$. In 
    this case $\downarrow\colon PX\to P^2X$ has components $\downarrow_C$ 
    taking a transformation $\theta\colon\mathbf yC\times X\Rightarrow\Omega$ 
    to one $\downset_C(\theta)\colon \mathbf yC\times PX\Rightarrow\Omega$ 
    components
        \begin{equation*}
            \theta_C(\theta)_D\colon\mscr C(D,C)\times PX(C)\to \Omega(C)
        \end{equation*}  
    with action 
        \begin{equation} \label{equation:downsetarrowactionforpowersetobject}
            \downset_C(\theta)_D(f,\tau) := (\tau\leq \theta\cdot f).
        \end{equation}
    Here the action $\theta\cdot f$ is essentially precomposition with $f$. 
    That is, the transformation $\theta\cdot f$ is the composite 
    $\theta(f-,-)\colon\mathbf yD\times X\Rightarrow\Omega$ having the action 
        \begin{equation*}
            (\theta\cdot f)_B(g,b) = \theta_B(fg,b)
        \end{equation*}
    where $(\theta\cdot f)_B\colon \mscr C(B,D)\times XB\to \Omega(B)$. Thus, 
    more precisely, $\tau\leq \theta\cdot f$ from the computation in 
    \cref{equation:downsetarrowactionforpowersetobject} is the sieve of arrows 
    $g\colon B\to D$ for which $\tau\cdot f \leq \theta \cdot fg$ holds where 
    the action is described as above. \qed
\end{example}

\begin{example}[Singleton Morphism]
    \label{example:completedescriptionofsingletonmorphisminpresheaftoposes}
    In the special case of the previous example where $r\colon R\to X\times X$ 
    is the diagonal $\Delta\colon X\to X\times X$ and $\leq_r$ is the identity 
    predicate $=_X$, the upset (and downset) morphisms are traditionally denote 
    by $\sigma_X\colon X\to PX$, thought of as the singleton morphism 
    $x\mapsto \lbrace x\rbrace$. In this case, the action is 
        \begin{equation*}
            \sigma_X(b)_D(f,a) = (a =_Xb\cdot f)
        \end{equation*} 
    using the description of \cref{equation:downsetactiononcomponents}. Of 
    course this result is a sieve, namely, the sieve of arrows $g\colon B\to D$ 
    for which $a\cdot g = b\cdot fg$ holds as a literal equality in $XB$. \qed
\end{example}

\begin{construction} \label{construction:powersetofupsetarrowdescription}
    Fix a relation $(X,r)$ in a presheaf topos $\mscr E = [\mscr C^{op},\set]$ 
    and let $\uparrow \colon X\to PX$ denote the upperset arrow from 
    \cref{example:completedescriptionofupsetanddownsetarrowsinpresheaftoposes}. 
    The arrow $P(\uparrow)\colon P^2X\to PX$ associated to $\uparrow$ under the 
    functor $P\colon \mscr E^{op}\to \mscr E$ is the unique one making the 
    diagram 
        \[\begin{tikzcd}
            {X\times PX} & \Omega \\
            {X\times P^2X} & {PX\times P^2X}
            \arrow["{\in_X}", from=1-1, to=1-2]
            \arrow["{1\times P(\uparrow)}", from=2-1, to=1-1]
            \arrow["{\uparrow\times 1}"', from=2-1, to=2-2]
            \arrow["{\in_{PX}}"', from=2-2, to=1-2]
        \end{tikzcd}\]
    commute. By this universal property, the construction of each $\in$-morphism, and naturality, it follows that each component 
    $P(\uparrow)_C(\theta)\colon \mathbf yC\times X\Rightarrow\Omega$ has the 
    action 
        \begin{equation*}
            P(\uparrow)_C(\theta)_D(f,a) = \theta_D(f,a)
        \end{equation*}
    where $\theta\colon \mathbf y C\times PX\Rightarrow\Omega$ and $f\colon 
    D\to C$ is an arrow and $a\in XD$ is any element. \qed
\end{construction}

\begin{example}[Necessity Operator] 
    \label{example:necessityoperatorinpresheaftoposes}
    Let $(X,r)$ again denote a relation in a presheaf category $\mscr E = 
    [\mscr C^{op},\set]$. The interior operator $\interior\colon PX\to PX$ can 
    be described in the following way on the basis of our computations of 
    $P(\uparrow)$ and $\downarrow$ above. That is, each component 
    $\interior_C\colon PX(C)\to PX(C)$ acts on transformation $\theta\colon 
    \mathbf yC\times X\Rightarrow\Omega$, yielding a new transformation 
    $\interior_C(\theta)$ of the same form. By the descriptions in 
    \cref{example:completedescriptionofupsetanddownsetarrowsinpresheaftoposes} 
    and \cref{construction:powersetofupsetarrowdescription} immediately above 
    and the definition $\interior = P(\uparrow)\circ \downarrow$, this is the 
    composite
        \[\begin{tikzcd}
            {\mathbf y C\times X} & {\mathbf yC\times PC} & \Omega
            \arrow["{1\times \uparrow}", Rightarrow, from=1-1, to=1-2]
            \arrow["{\downarrow_C\!(\theta)}", Rightarrow, from=1-2, to=1-3]
        \end{tikzcd}\]
    whose components 
        \begin{equation*}
            \interior_C(\theta)_D\colon \mathscr C(D,C)\times XD\to \Omega(D)
        \end{equation*}
    have action 
        \begin{equation*}
            \downset_C(\theta)_D(f,\upset_D(a)) = (\upset_D(a)\leq \theta\cdot f)
        \end{equation*}
    specifically by \cref{equation:downsetarrowactionforpowersetobject} where 
    $f\colon D\to C$ is any arrow and $a\in XD$ is any element. So, more 
    colloquially, the value of the $D$-component of $\interior_C(\theta)$ on 
    such a pair $(f,a)$ is the sieve of arrows $g\colon B\to D$ such that 
    $\uparrow_D\!\!(a)\cdot g \leq \theta\cdot fg$ holds in $\Omega(B)$. This 
    is evidently a statement that the sieve on the left side of the inequality 
    is contained in the sieve on the right. In the special case that we have 
    $\phi\colon X\to \Omega$ and we consider its transpose 
    $\hat\phi\colon 1\to PX$, each such component therefore has the action
        \begin{equation*}
            \downset_C(\hat\phi(\ast))_D(f,\upset_D(a)) = (\upset_D(a)\leq \hat\phi_C(\ast)\cdot f).
        \end{equation*}  
    In particular, both sides of the inequality on the right are transformations 
    $\mathbf yD\times X\Rightarrow\Omega$ and $\leq$ is the order on $PX$. The 
    value of the inequality is the sieve of arrows $g\colon B\to D$ such that 
    $\upset_D(a)\cdot g\leq \hat\phi_C(\ast)\cdot fg$ holds in $\Omega (B)$. 
    Each side of the last inequality are transformations 
    $\mathbf yB\times X\to \Omega$, which on components $A$ for any 
    $h\colon A\to B$ and $b\in XA$ have the actions 
        \begin{equation*}
            (\upset_D(a)\cdot g)_A(h,b) = (a\cdot gh \leq_r b)
        \end{equation*}
    in the order on $X$ by \cref{equation:upsetactiononcomponents}; and also
        \begin{equation*}
            (\hat\phi_C(\ast)\cdot fg)_A(h,b) = \phi_A(b)
        \end{equation*}
    by \cref{equation:actionofphihatinpresheaftopos}. So, in summary, what we 
    have derived is the statement that the sieve $(a\cdot gh \leq_r b)$ as a 
    set must be contained in the sieve $\phi_A(b)$ for all appropriate elements 
    $a, b$ and arrows $g, h$. We can now derive the forcing semantics of $\Box$ 
    in the presheaf topos directly. By transposing and just looking at the 
    diagonal component where $C=D$, we end up with the membership statement
        \begin{equation*}
            (1_X\in \upset_C(a)) \leq (1_X \in \hat\phi_C(\ast)\cdot f).
        \end{equation*}
    The $\leq$ is equivalently the implication operator on $\Omega$, so if we 
    ask that this is true of an element $b\in XC$, we have that the implication
        \begin{equation*}
            b\in \upset_C(a) \text{ implies } b\in \hat\phi_C(\ast)\cdot f
        \end{equation*}
    holds by usual forcing semantics. But by transposing, the antecedent is 
    $a\leq b$ in the order on $XC$, while the consequent is $\phi_C(b)$. This 
    yields the necessity forcing statement in 
    \cref{proposition:presheafforcingsemanticsinteriorandclosure}. \qed
\end{example}

\subsection{The Indistinguishability Interpretation}
\label{subsection:indistinguishability}

This subsection will give an epistemic interpretation to the forcing statements 
of the previous development. Partly this is motivated by \emph{epistemic modal 
logic} and partly this is suggested by the histories presheaf itself. For on 
the one hand epistemic modal logic is modeled by Kripke frames and its 
axiomatics are the systems (S4.2) and (S5). On the other hand, the histories 
presheaf, involving histories and events, lends itself, by an imaginative leap, 
to the view that propositions might be true or \emph{verifiable} from the 
perspective of a particular event and in a certain history. In fact, our 
semantics are precisely that a space-time event witnesses a necessary 
proposition in a particular history if, and only if, it witnesses the 
proposition itself in any other history obviously undivided at that event. So, 
in a manner of speaking, a proposition is necessarily true in a particular 
history at an event if, and only if, the proposition appears to be the case 
from the perspective of the event in any history indistinguishable from the the 
one in question.

Epistemic logic \cite{hintikka1962}, \cite{ditmarsch2015} is the modal logic of 
knowledge and belief. There are usually two operators $K_a\phi$ and $B_a\phi$ 
glossed as \emph{agent $a$ knows that} $\phi$ and \emph{agent $a$ believes 
that} $\phi$. At least $K$ tends to be asked to satisfy the comonad axioms. For 
this reason, there is an analogy between knowledge and modal necessity. In fact 
\cite[\S 2.3]{baltagetal2015} makes this connection precise: topological 
interior models knowledge and open sets model evidence. Belief is then modeled 
by the closure of the interior \cite[\S 3.2]{baltagetal2015}. When $K$ is asked 
to have a dual operator, usually classical (S5) provides the associated 
axiomatics. Equivalence relations then comprise the natural class of Kripke 
models. The \emph{indistinguishability interpretation} of modal epistemic 
statements originated with \cite{lehmann1984}. In this set-up, there is a 
family of equivalence relations $\sim_i$ parameterized by epistemic agents $i$ 
of the model. Quoting \cite[\S]{lehmann1984} directly:
    \begin{quote}
        ``The intuitive meaning of $\sim_i$ is the following: two histories $\sigma$ and $\tau$ are equivalent with respect to $i$ and $k$ if the knowledge that person $i$ has gathered about the world and its history up to instant $k$ cannot enable him to distinguish between $\sigma$ and $\tau$. In other words if the real history is $\sigma$, then, for all that person $i$ knows at instant $k$, the real history could as well be $\tau$.''
    \end{quote}
So, a modal statement $K_i\phi$ is true for a history $\sigma$ if, and only if, 
$\phi$ is true in every history $\tau$ related to $\sigma$ under the provided 
equivalence relation for that agent. In other words, as in the quote, an agent 
is said to know that $\phi$ in a particular history if $\phi$ appears to be the 
case in every history indistinguishable from it.

In the following, we thus allow ourselves to think of the parameterizing 
category $\mscr C$ as a system of agents and their relations. We think of the 
values $XC$ associated to agents $C\in\mscr C$ as sets of histories to which 
they have access or belong. Specializing to the case of our world $\mscr W$, 
this is exactly what the histories presheaf does. And it is not too much of a 
stretch to view point events as agents since anyway observers have definite 
space-time coordinates in a given inertial frame. There is further warrant for 
it in the intuitive interpretation of the intuitionistic semantics of 
\cite[\S 1.1]{kripke1963}, which we quote as follows:
    \begin{quote}
        ``We intend the nodes $\mathbf H$ to represent points in time (or ``evidental situation''), at which we may have various pieces of information. If, at a particular point $\mathbf H$ in time, we have enough information to prove a proposition $A$, we say that $\phi(A,\mathbf H) = \mathbf T$; if we lack such information, we say that $\phi(A,\mathbf H)=\mathbf F$. If $\phi(A,\mathbf H)=\mathbf T$ we can say that $A$ has been \emph{verified} as the point $\mathbf H$ in time; if $\phi(A,\mathbf H) = \mathbf F$, then $A$ has \emph{not been verified} at $\mathbf H$.''
    \end{quote}
In other words, the nodes, i.e. objects, of the underlying system, i.e. 
category, are moments in (space-) time which represent or have attached to 
them, or in which are available, a situation or configuration of information 
that serves as evidence for the potential confirmation of a proposition under 
consideration. It should be noted that in the quote $\phi$ stands for the 
Kripke model itself associating $\mathbf T$ or $\mathbf F$ to the proposition 
$A$ and the stage $\mathbf H$. So, that its value is $\mathbf F$ means only 
that $A$ has not been verified at $\mathbf H$ (given the information available 
at that instant), but leaves open the possibility of being verified at some 
later stage. For this reason, an $\mathbf F$ does not count as disconfirmation 
of a proposition at $\mathbf H$. 

In any case, we shall proceed as though point events $e$ in branching 
space-time $\mscr W$ represent, or correspond to, \emph{evidential situations}, 
potentially available to any agent or observer occupying that space-time 
position, by means of which propositions $\phi\colon X\to \Omega$ may be 
verified. Histories of course are the elements of the values of the histories 
presheaf, so we must take account of alternatives and branching as well: for a 
given proposition $\phi\colon H\to \Omega$, we take $\phi_e(h)$ to represent 
the value of whether $\phi$ is or has been \emph{verified at $e$ in $h$}. 
Perhaps better, since the logic of topos is not Boolean, is that it is the 
\emph{extent to which} such a verification has been made. This viewpoint 
comports well with the indistinguishability interpretation. As above, we denote 
by $\Box\phi$ the morphism 
    \begin{equation*}
        \Box\phi := \interior (\hat\phi)\colon 1\to PH
    \end{equation*}
where $\hat\phi\colon 1\to PH$ is the name of $\phi$ given by exponential 
transpose. By interior semantics in \cref{proposition:interiorsemantics} and \cref{proposition:presheafforcingsemanticsinteriorandclosure} we 
then have
    \begin{align*}
        e\Vdash \Box\phi(h) &\text{ if, and only if, } e\Vdash h\in_H\interior(\hat\phi) \\
                            &\text{ if, and only if, } e\Vdash h\approx_e h' \text{ implies } e\Vdash h'\in_H\hat\phi \\
                            &\text{ if, and only if, } e\Vdash h\approx_e h' \text{ implies } e \Vdash \phi(h').
    \end{align*}
Thus, our gloss on $e\Vdash \Box\phi(h)$ is the statement that a point event
$e$ forces \emph{necessarily} $\phi$ at a history $h$ containing $e$ if, and 
only if, $e$ forces $\phi$ at every history $h'$ obviously undivided from $h$ 
at $e$. More colloquially, then, $\phi$ is necessarily true of $h$ at an event 
$e$ if, and only if, it is true of all histories obviously undivided from $h$ 
at $e$. Interpreting \emph{obviously undivided} as the relationship encoding 
the fact that $e$ does not, or cannot, distinguish between any two histories 
for which $e$ is not a choice point, the statement $e\Vdash \Box\phi(h)$ thus 
says that the proposition $\phi$ is necessarily true of $h$ at the event $e$ 
if, and only if, it is true of any history indistinguishable from $h$ at $e$. 
Likewise for the possibility operator, we shall denote by $\Diamond \phi$ the 
morphism 
    \begin{equation*}
        \Diamond\phi:=\closure(\hat\phi)\colon 1\to\ PH.
    \end{equation*}
By the closure semantics of  \cref{corollary:closureforcingsemantics}, we have 
then that 
    \begin{align*}
        e\Vdash \Diamond\phi(h) &\text{ if, and only if, } e\Vdash h\in_H\closure(\hat\phi) \\
                                &\text{ if, and only if, } e\Vdash h \approx_e h' \text{ and } e \Vdash h'\in_H\hat\phi \text{ for some } h' \\
                                &\text{ if, and only if, } e\Vdash h\approx_eh' \text{ and } e\Vdash \phi(h') \text{ for some } h'.
    \end{align*}
That is, $e$ forces \emph{possibly} $\phi$ at $h$ if, and only if, there is 
some history $h'$ indistinguishable from $h$ at $e$ of which $\phi$ is true at 
stage $e$. So, to distinguish from necessity, possibility is just the statement 
that the proposition is true of some history, potentially distinct from the one 
under consideration, but which is indistinguishable from it at the current 
stage, of which the proposition is true at the current stage.

\subsection{First-Order Considerations}
\label{subsection:firstorderconsiderations}

In this final subsection, we begin the project of examining first-order 
formulas in light of the semantics developed so far. Our first question is 
whether our models validate the \emph{Barcan formulas} of \cite{barcan_1947}. 
Two of these are discussed below in \cref{corollary:barcanformula} and the 
subsequent examples. We shall look finally at the existence predicates of 
\cite{scott1979} in light of these results and the actualist thesis 
\cref{equation:actualist-thesis}. Throughout we shall use the recipe for 
formation of quantified formula in toposes as discussed in 
\cref{example:internalquantifiersasadjoints}. Our main tool will be the 
presheaf semantics in 
\cref{proposition:presheafforcingsemanticsinteriorandclosure}.

The first result shows that in any topos, the interior operator on the 
subobject classifier is trivial.

\begin{prop} \label{proposition:NecessityOnOmegaIsTrivial}
    In any topos, $\interior \colon \Omega\to \Omega$ is the identity.
\end{prop}
\begin{proof}
    The downset arrow $1\to P1$ is the generic subobject 
    $\top\colon 1\to P1=\Omega$. In the following diagram, the top square 
    commutes by the definition of $P(\uparrow)$ as a transpose, while the 
    triangle commutes by the definition of $\downarrow$ also as a transpose:
        \[\begin{tikzcd}
            {1\times P1} && \Omega \\
            {1\times P\Omega} & {\Omega\times P\Omega} & \Omega \\
            {1\times \Omega} & \Omega\times\Omega
            \arrow["{1\times P(\uparrow)}", from=2-1, to=1-1]
            \arrow["{1\times \downarrow}", from=3-1, to=2-1]
            \arrow["{\top\times 1}"', from=3-1, to=3-2]
            \arrow["{1\times \downarrow}", from=3-2, to=2-2]
            \arrow["{\in_\Omega}", from=2-2, to=2-3]
            \arrow[Rightarrow, no head, from=2-3, to=1-3]
            \arrow["{\Rightarrow_{\Omega}}"'{pos=0.6}, curve={height=6pt}, from=3-2, to=2-3]
            \arrow["{\top\times 1}", from=2-1, to=2-2]
            \arrow["{\in_1}", from=1-1, to=1-3]
        \end{tikzcd}.\]
    On the whole, this realizes $P(\uparrow)\circ\downarrow$ as the transpose 
    of $\top\Rightarrow_\Omega 1$. But it is easy to see that this arrow must 
    be the identity on $\Omega$, for the square in the diagram
        \[\begin{tikzcd}
            1 & {(\leq)} & 1 \\
            1\times\Omega & \Omega\times\Omega & \Omega
            \arrow[from=1-1, to=1-2]
            \arrow[from=1-2, to=2-2]
            \arrow["{1\times \top}"', from=1-1, to=2-1]
            \arrow["{\top\times 1}"', from=2-1, to=2-2]
            \arrow["{!}", from=1-2, to=1-3]
            \arrow["\top", from=1-3, to=2-3]
            \arrow["{\Rightarrow_\Omega}"', from=2-2, to=2-3]
        \end{tikzcd}.\]
    are both pullbacks. The rightmost is a pullback by the definition of the 
    implication operator; the leftmost is one by a direct computation. Since 
    $(\top\Rightarrow 1) = 1$, this shows the result by uniqueness.
\end{proof}

\begin{cor} \label{corollary:barcanformula}
    In any topos, with $\Box$ interpreted as interior, any preordered object $(X,r)$ validates two of the so-called \emph{Barcan Formulas}, namely,  
        \begin{enumerate}
            \item $\forall x \Box\phi(x) \supset \Box\forall x\phi(x)$
            \item $\exists x \Box\phi(x) \supset \Box\exists x\phi(x)$.
        \end{enumerate}
\end{cor}
\begin{proof}
    We shall prove the first one; that of the second is analogous. 
    \cref{proposition:NecessityOnOmegaIsTrivial} shows that for any truth value 
    $\omega\colon 1\to \Omega$, we have $\Box\omega = \omega$. In particular, 
    this is the case if $\omega$ is any closed formula. Thus, 
    $\Box\forall_X\phi = \forall_X\phi$ holds for any $\phi\colon X\to \Omega$ 
    in the topos. It therefore suffices to prove that 
    $\forall_X\Box\phi\Rightarrow_{\Omega}\forall_X\phi$ holds, or equivalently 
    that $\forall_X\Box\phi\leq \forall_X\phi$ holds in the order on $\Omega$, 
    but this follows immediately. For $\Box\phi\leq \phi$ always holds by 
    \cref{lemma:reflexivityproof} and $\forall_X$, being a right adjoint, is in 
    particular a functor, that is, order-preserving.
\end{proof}

\begin{remark}
    With regard to existential quantification, in fancier language, this 
    corollary implies that the \emph{de re} modality $\exists x\Box\phi$ 
    implies the \emph{de dicto} modality $\Box\exists x\phi$ in the sense that 
    any preordered object $(X,r)$ in a topos validates the implication 
    $\exists x\Box\phi \supset \Box \exists \phi$. This is expected; as 
    remarked in \cite[\S 4.1]{awodeykishida2008} this formula is provable in 
    first-order S4, hence it is valid in all topological semantics including 
    that arising from the Alexandrov topology. The point is that this 
    Alexandrov topology result caries over into any topos. What is suprising is 
    that the first formula $\forall x \Box\phi(x) \supset \Box\forall x\phi(x)$ 
    is validated in any such model. For in \cite[\S 4.1]{awodeykishida2008}, 
    there is a sheaf-theoretic counter model. Thus, our result shows that this 
    Barcan formula is true for a restricted class of topological models and 
    holds in any such model in any topos. Concerning the converses, that of the 
    first formula is provable in first-order S4, whereas that of the second has 
    a counter-model.\qed
\end{remark}

\begin{example} 
    \label{example:derededictoconversefalse}
    We shall invalidate $\Box\exists x\phi\supset\exists x\Box\phi$ in the 
    category of sets. Let $\mscr E = [\mathbf 1,\set] = \set$ and consider the 
    set $\mathbb N$ of natural numbers viewed as a poset with its usual 
    ordering. Let $\phi\colon \mathbb N \to \mathbf 2$ denote the proposition 
    \emph{is even}. In $\set$, the operators $\Box$ and $\Diamond$ have their 
    usual description as Alexandrov interior and closure from 
    \cref{equation:interiorAlexandrovTopologyElementDescription} and 
    \cref{equation:closureAlexandrovTopologyElementDescription}. In this case, 
    $\Box\phi$ is the set 
        \begin{equation*}
            \Box \phi = \lbrace n\in \mathbb N \mid \text{for all } m\in\mathbb N \text{ if } n\leq m \text{ then } m\in \phi \rbrace
        \end{equation*} 
    which is of course equal to $\emptyset$ since for any given natural number, 
    there is an odd one greater than or equal to it. Additionally, we can use 
    the computation of `$\exists$' from \cite[\S I.9]{maclanemoerdijk1992} to 
    describe $\exists_\mathbb N\colon P\mathbb N \to \mathbf 2$. In the case of 
    our $\phi\subset\mathbb N$, we have
        \begin{equation*}
            \exists_\mathbb N \phi = \lbrace y \in \mathbf 1 \mid \text{there is } n \in\mathbb N \text{ such that } n\in \phi  \rbrace = 1
        \end{equation*}
    since there is an even natural number and
        \begin{equation*}
            \exists_\mathbb N\Box\phi = \lbrace y \in\mathbf 1 \mid \text{there is } n\in \mathbb N \text{ such that } n \in \Box\phi \rbrace = 0
        \end{equation*}
    since again for every natural number there is an odd natural number greater 
    than or equal to it. So, the \emph{de re} statement will imply the \emph{de 
    dicto} since $0\leq 1$, but the converse, namely, 
    $\exists_\mathbb N\phi \leq \exists_\mathbb N\Box\phi$ is impossible. \qed
\end{example}

Likewise, the possibility operator on $\Omega  = P1$ is trivial. Recall that 
the closure operator appears in  \cref{definition:internalclosureoperator} as 
$\closure = \bigvee\exists_\downarrow$. Thus, the proof of the next result 
relies partly upon  \cref{construction:internaljoin} which developed the 
internal join operator.

\begin{prop} \label{proposition:possibilitytrivialonOmega}
    In any topos $\closure \colon \Omega\to \Omega$ is the identity morphism.
\end{prop}
\begin{proof}
    First note that $\exists_\downarrow$ is actually the transpose of the meet 
    $\wedge\colon \Omega\times\Omega\to\Omega$. This is because we are looking 
    at the order on the terminal, $1$, which is the isomorphism 
    $1\cong 1\times 1$. As a result, $\downarrow\colon 1\to P1$ as in the 
    diagram 
        \[\begin{tikzcd}
            {1\times P1} & \Omega \\
            {1\times 1}
            \arrow["{\in_1}", from=1-1, to=1-2]
            \arrow["1\times\downarrow", from=2-1, to=1-1]
            \arrow["\leq"', curve={height=6pt}, from=2-1, to=1-2]
        \end{tikzcd}\]
    must be $\top\colon 1\to \Omega$. So, $\exists_\downarrow = \exists_\top$ 
    is formed as in \cref{subsection:toposesandinternallogic}. That is, first 
    take the classifying morphism as on the bottom of the diagram:
        \[\begin{tikzcd}
            1 & 1 \\
            {1\times P1} & \Omega \\
            {P1\times P1} & \Omega
            \arrow[Rightarrow, no head, from=1-1, to=1-2]
            \arrow["\top", from=1-2, to=2-2]
            \arrow["{\langle1,\top\rangle}"', from=1-1, to=2-1]
            \arrow["{\langle \top,1\rangle}"', from=2-1, to=3-1]
            \arrow["\wedge"', dashed, from=3-1, to=3-2]
            \arrow[Rightarrow, no head, from=2-2, to=3-2]
            \arrow["{\in_1}", from=2-1, to=2-2]
        \end{tikzcd}\]
    Note that $\in_1$ is an isomorphism, meaning that $\langle 1,\top\rangle$ 
    is the classified subobject; likewise $\wedge$ is \emph{defined} as the 
    classifying morphism of $\langle \top,\top\rangle$. Now, 
    $\exists_\downarrow = \exists_\top$ is therefore the transpose of $\wedge$, 
    as claimed. Now, put this together with the construction of 
    $\bigvee\colon P\Omega\to \Omega$. We have a diagram
        \[\begin{tikzcd}
            {1\times P1} & \Omega & \Omega \\
            {1\times P\Omega} & {\Omega\times P\Omega} & \Omega \\
            {1\times P1} & {\Omega\times P1}
            \arrow["{\in_1}", from=1-1, to=1-2]
            \arrow[Rightarrow, no head, from=1-2, to=1-3]
            \arrow["1\times\bigvee", from=2-1, to=1-1]
            \arrow["{\top\times 1}", from=2-1, to=2-2]
            \arrow["{\in_\Omega}", from=2-2, to=1-2]
            \arrow["{\in_\Omega}", from=2-2, to=2-3]
            \arrow[Rightarrow, no head, from=2-3, to=1-3]
            \arrow["{\top\times 1}"', from=3-1, to=3-2]
            \arrow["{1\times\exists_\top}", from=3-2, to=2-2]
            \arrow["{1\times\exists_\top}", from=3-1, to=2-1]
            \arrow["\wedge"', curve={height=12pt}, from=3-2, to=2-3]
        \end{tikzcd}\]
    Note that the top-left square commutes by the construction of $\bigvee$ as 
    $\Omega^\top$. But the point of presenting this diagram is that 
    $\top\wedge 1 = 1$ in any topos. This is well-known in propositional logic, 
    but for the record, the proof is easy, just noting that the following is a 
    pullback:
        \[\begin{tikzcd}
            1 & 1 & 1 \\
            {1\times \Omega} & \Omega\times\Omega & \Omega
            \arrow[Rightarrow, no head, from=1-1, to=1-2]
            \arrow[Rightarrow, no head, from=1-2, to=1-3]
            \arrow["{\langle 1,\top\rangle}"', from=1-1, to=2-1]
            \arrow["\top\times1"', from=2-1, to=2-2]
            \arrow["\wedge"', from=2-2, to=2-3]
            \arrow["\top", from=1-3, to=2-3]
            \arrow["{\langle\top,\top\rangle}", from=1-2, to=2-2]
        \end{tikzcd}.\]
    In any case, by uniqueness of transposes, this proves that 
    $\bigvee\exists_\top =1$ must hold.
\end{proof}

\begin{cor} 
    \label{corollary:barcanredux}
    Any preordered object $(X,r)$ in a topos $\mscr E$ validates the formulas 
        \begin{enumerate}
            \item $\Diamond\forall x\phi\supset \forall x\Diamond\phi$
            \item $\Diamond\exists x\phi\supset \exists x\Diamond\phi$
        \end{enumerate}
\end{cor}
\begin{proof}
    The proof is analogous to that of  \cref{corollary:barcanformula} again 
    using the previous proposition, the functoriality of quantifiers, and the 
    fact that the order validates $\phi\supset\Diamond \phi$.
\end{proof}

\begin{remark}
    The converses of the formulas in \cref{corollary:barcanredux} are also of 
    interest. These are $\forall x\Diamond\phi\supset \Diamond\forall x\phi$ 
    and $\exists x\Diamond\phi\supset \Diamond\exists x\phi$. In general, the 
    former fails in our models in toposes whereas the latter is provable in 
    first-order S4 \cite[\S 4.1]{awodeykishida2008}. In our system, the latter 
    takes the form $\exists_X\Diamond\phi \leq \exists_X\phi$ in the order on 
    $\Omega$ since $\Diamond = 1$ on $\Omega$ by 
    \cref{proposition:possibilitytrivialonOmega}. Its validity might be 
    surprising, since it glosses as \emph{if there is something for which a 
    proposition is possibly the case then there is something for which it is 
    the case}, which we are tempted to sloganize as \emph{possibilities are 
    actual}, or perhaps \emph{possibilities are real}. The following example 
    invalidates the universally quantified formula and sheds some light on the 
    truth of the existential one. \qed
\end{remark}

\begin{example} \label{example:invalidateBarcanconverse}
    Return to the situation of \cref{example:simplesbranchingspacetime}, 
    namely, the simple branching space-time $\mscr W$ with four point events 
    and three nontrivial arrows. Again, the histories presheaf $H$ makes the 
    assignments summarized in the diagram 
        \[\begin{tikzcd}
            && {\lbrace h_1\rbrace} && {\lbrace h_1\rbrace} \\
            H & {=} && {\lbrace h_1,h_2\rbrace} \\
            &&& {\lbrace h_1,h_2\rbrace}
            \arrow["{-\cdot\alpha}"', hook, from=1-3, to=2-4]
            \arrow["{-\cdot\beta}", hook', from=1-5, to=2-4]
            \arrow["{-\cdot\gamma}", from=2-4, to=3-4]
        \end{tikzcd}.\]
    Most crucially, $e_0$ is the unique choice point, and $h_1$ and $h_2$ are 
    obviously undivided at $e_{-1}$ as discussed in 
    \cref{example:simplesbranchingspacetimeredux}. Let $H_1$ and $H_2$ denote 
    the following subobjects of $H$, namely, those summarized by their images 
    in $\set$ in the respective diagrams
        
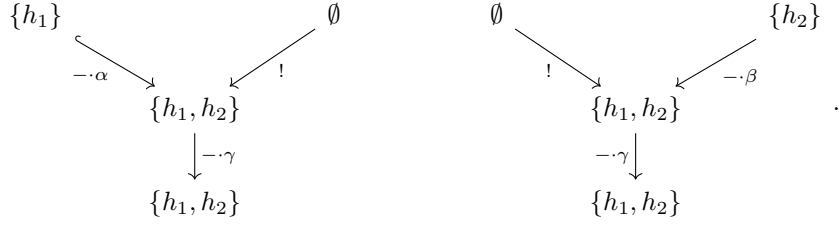
\[\begin{tikzcd}
            {\lbrace h_1\rbrace} && \emptyset && \emptyset && {\lbrace h_2\rbrace} \\
            & {\lbrace h_1,h_2\rbrace} &&&& {\lbrace h_1,h_2\rbrace} \\
            & {\lbrace h_1,h_2\rbrace} &&&& {\lbrace h_1,h_2\rbrace}
            \arrow["{-\cdot\alpha}"', hook, from=1-1, to=2-2]
            \arrow["{!}", from=1-3, to=2-2]
            \arrow["{-\cdot\gamma}", from=2-2, to=3-2]
            \arrow["{!}"', from=1-5, to=2-6]
            \arrow["{-\cdot \beta}", from=1-7, to=2-6]
            \arrow["{-\cdot\gamma}"', from=2-6, to=3-6]
        \end{tikzcd}.\]
    Let $\phi\colon H\to \Omega$ denote the classifying morphism of $H_1$, the 
    subobject on the left. Think of this as the proposition \emph{$h_1$ 
    occurs}. Our claim is that the formula $\forall x\Diamond \phi \supset 
    \forall x\phi$ is not valid in this model. This formula is interpreted as 
    an implication, or equivalently, as an inequality in the order on $\Omega$; 
    thus, the implication is validated if there is a valid containment between 
    (the subobjects associated to) the antecedent $\forall_H\Diamond\phi$ and 
    the consequent $\forall_H\phi$. It takes only one component where this 
    fails toevidence that the implication cannot hold. So, consider the values 
    at the event $e_{-1}$. By the computation of the effect of internal 
    universal quantification in presheaf categories (cf. 
    \cite[Example 1.3]{awodeykishidakotzsch2014} and \cite[Equations III.8.(14)-III.8.(15)]{maclanemoerdijk1992}), we have
        \begin{equation*}
            (\forall_H\phi)_{e_{-1}} = \bigvee \lbrace s\in\Omega(e_{-1}) \mid u^*s\leq \phi_{e_{-1}}(u,b) \text{ for all } u\colon y \to e_{-1}, b\in H_y\rbrace
        \end{equation*}
    and 
        \begin{equation*}
            (\forall_H\Diamond\phi)_{e_{-1}} = \bigvee \lbrace s\in\Omega(e_{-1}) \mid u^*s\leq \Diamond\phi_{e_{-1}}(u,b) \text{ for all } u\colon y \to e_{-1}, b\in H_y\rbrace.
        \end{equation*}
    There are two sieves on $e_{-1}$, namely, $\displaystyle \Omega(e_{-1}) = 
    \lbrace \top_{e_{-1}},\emptyset\rbrace$. Moreover, there is only one arrow 
    $y\to e_{-1}$, namely, the identity. So, we have that $u^*s=s$ is either 
    $\top$ or $\emptyset$ and consequently
        \begin{equation*}
            (\forall_H\phi)_{e_{-1}} = \bigvee \lbrace s\in\Omega(e_{-1}) \mid s\leq \phi_{e_{-1}}(h) \text{ for all } h\in H_{e_{-1}}\rbrace
        \end{equation*}
    and 
        \begin{equation*}
            (\forall_H\Diamond\phi)_{e_{-1}} = \bigvee \lbrace s\in\Omega(e_{-1}) \mid s\leq \Diamond\phi_{e_{-1}}(h) \text{ for all } h\in H_{e_{-1}}\rbrace.
        \end{equation*}
    But on the one hand, $\phi_{e_{-1}}(h_2) = \emptyset$ holds by construction 
    of $\phi$, meaning that $\top$ is not in the former union. On the other 
    hand, $\Diamond \phi_{e_{-1}}(h_1) = \Diamond \phi_{e_{-1}}(h_2) = \top$ 
    both hold by the presheaf forcing semantics of $\Diamond$ in  \cref{proposition:presheafforcingsemanticsinteriorandclosure} since 
    $h_1\approx_{e_{-1}} h_2$ holds. Therefore, $\top$ is in the latter union 
    and so we conclude that
        \begin{equation*}
            (\forall_H\phi)_{e_{-1}} = \emptyset \qquad (\forall_H\Diamond\phi)_{e_{-1}} = \top
        \end{equation*}
    Therefore, the latter is not contained in the former, and we have our 
    counter-model of $\forall x\Diamond \phi \supset \forall x\phi$ as was 
    claimed. Note that this is the simplest component of the involved presheaf 
    that invalidates the formula. But the components at $e_2$ do the job too. 
    For, using the computations of the effect of internal `$\forall$' above, 
    one can see that
        \begin{equation*}
            (\forall_H\phi)_{e_2} = \emptyset \qquad \text{whereas} \qquad (\forall_H\Diamond \phi)_{e_2} = \lbrace \beta\gamma\rbrace
        \end{equation*}
    since again $h_1$ and $h_2$ are obviously undivided at $e_{-1}$. \qed
\end{example}

A final application now concerns the \emph{existence predicate} of, for 
example, \cite{scott1979}. This is indicated by the symbol `$\mathsf E$' 
applied to terms $t$, and resulting in a predicate $\mathsf E(t)$ having the 
intuitive meaning of \emph{$t$ exists}. This predicate is used in the reference 
in the formulation of theories requiring partial operations such as the 
theories of categories and of local rings. For simplicity\footnote{See 
\cite[Metatheorem 2.3]{scott1979} for details on the relationship between 
equality, equivalence and existence in this system.}, we shall take the 
existence prediate to be defined as
    \begin{equation*}
        \mathsf E(t) \leftrightsquigarrow  \exists y(t = y)
    \end{equation*}
for a term of type $X$, a variable of type $X$. This is also an 
\textit{actualist definition} of the existence predicate 
\cite[\S 2.2.2]{menzel2022}. Note that $\mathsf E(t)$ is thus $\Omega$-valued 
but not necessarily a truth value. Hence it is not the case that $\Box\mathsf 
E(t) = \mathsf E(t)$. 

In the indistinguishability interpretation, we have the following. 
In any presheaf topos $\mscr E = [\mscr C^{op},\set]$, the forcing semantics of 
existential quantification for an preordered object $(X,r)$ yields
    \begin{equation*}
        C\Vdash \exists x( a =_X x) \text{ if, and only if, } C\Vdash a=_Xb \text{ for some } b\in XC.
    \end{equation*}
But the right side is just ordinary equality in $XC$. Hence we have 
    \begin{equation*}
        C\Vdash \exists x( a =_X x) \text{ if, and only if, } a=b \text{ for some } b\in XC
    \end{equation*}
which is always true, since the forcing statement assumes that we are provided 
with the $a\in XC$. On the other hand, the formula $\exists x\Box (y=_Xx)$ has 
the forcing semantics
    \begin{align}
        C\Vdash \exists x\Box (a=_Xx) &\text{ if, and only if, } C\Vdash \exists x\Box (a=_Xx) \\
                                      &\text{ if, and only if, } C\Vdash \Box (a=_Xb) \text{ for some } b\in XC \\
                                      &\text{ if, and only if, } C\Vdash (b\approx b') \text{ implies } C\Vdash (a=_Xb') \text{ for all } b'\in XC
    \end{align}
by \cref{proposition:presheafforcingsemanticsinteriorandclosure}. The gloss on 
this is thus that there exists something necessarily equal to $a$ if, and only 
if, there is something $b$ such that $a$ is equal to everything in the 
equivalence class of $b$ under the relation on $X$. In the epistemic 
interpretation, there is something necessarily equal to $a$ if, and only if, 
there is something $b$ such that $a$ is equal to everything indistinguishable 
from $b$. But again $=_X$ reduces to ordinary equality when applied in each set 
$XC$ since $=_X$ classifies the diagonal. Hence we conclude that 
    \begin{equation*}
        C\Vdash \exists x\Box (a=_Xx) \text{ if, and only if, } C\Vdash (b\approx b') \text{ implies } C\Vdash (a=b') \text{ for all } b'\in XC
    \end{equation*}
which we think of as a very strong kind of indistinguishability: there exists 
something $b$ necessarily equal to $a$ if and only if $a$ is equal to 
everything indistinguishable from $b$. In general this will not be true in 
presheaf toposes as long as $\approx$ is a non-trivial relation, such as a 
genuinely objectively indeterminate branching space-time where the 
quantification is over histories $h$. In general this should not be surprising, 
for given the relation, there is no saying that $a$ is even equal to everything 
indistinguishable from itself. Generally, that is, identity implies 
indistinguishability but not conversely, unless we are in a special system 
where equivalence implies identity. Perhaps this is of interest in more general 
toposes such as genuine sheaf categories, which we leave to later developments. 

Finally, we observe that the previous developments invalidate actualist thesis 
\cref{equation:actualist-thesis}, namely, that \textit{existent potentialities 
are actual}. More colloquially, this says that what exists possibly or 
potentially exists actually. We formalize this as the formula     
    \begin{equation*}
        \forall x (\Diamond \mathsf E(x)\supset \mathsf E(x))
    \end{equation*}
which is of couse equal to 
    \begin{equation*}
        \forall x \Diamond \mathsf E(x)\supset \forall x\mathsf E(x)
    \end{equation*}
since in any topos $\forall_X$ is an internal functor. But this has a 
counter-model in the form on non-trivial branching above.

\begin{example}
    The result of \cref{example:invalidateBarcanconverse} we can apply to the special case of $\phi := \mathsf E(x)$ which shows that in the simple example of non-trivial branching the antecedent of $\forall x \Diamond \mathsf E(x)\supset \forall x\mathsf E(x)$ does not imply the consequent. We interpret this as a denial of the actualist thesis that possible existents are actual. Note that this does not necessarily validate possibilism. \qed
\end{example}

\section{Prospectus}

The results and perspectives in this paper offer a starting point for some 
further avenues of inquiry and raise some questions which we address here 
briefly in closing.

\begin{enumerate}
    \item The underlying poset \emph{our world} has been thought of as imposing
    a causal rather than temporal relation on point events. However, there is 
    no specific reason to restrict the general formalism to only this case. 
    That is, supposing that the underlying order is not necessarily 
    anti-symmetric, we have a reasonable model of temporally-related events. 
    Histories then are temporal histories rather than causal histories. 
    \emph{Obviously undivided} is still some kind of relation. Supposing 
    reasonable conditions under which it is symmetric, we have again an 
    internal equivalence relation which may again be given an epistemic 
    interpretation. Thus, on the one hand, we have a temporal parameterization 
    along with an object modeling an epistemic \emph{indistinguishability} 
    relation. The extent to which this promises a unified framework for a 
    treatment of epistemic and temporal modal logic \cite{dixon2015} will be 
    left to later investigations.
    \item The relation \emph{obviously undivided} defined on the histories 
    presheaf could be viewed as an object of the bicategory, or better the 
    double category, of relations in the topos of presheaves on \emph{our 
    world}. A more \emph{relational view} would be consistent with the 
    perspective of \cite{kishida2017} which views modalities as certain kinds 
    of relations and the adjunctions involving modalities and quantifiers are 
    studied there in some detail. Many of these features are formally 
    replicable in the setting of a \emph{double category of relations} 
    \cite{lambert2022b} and thus the question arises as to the extent to which 
    the essential constructions of this work are recovered in a more abstract 
    double-categorical framework. Any progress on this question would need to 
    be compared to, and probably motivated by, the direct bicategorical 
    treatment of modality appearing in \cite{hermida2011}. Motivating examples 
    in that study are bicategories of relations and of partial maps, both of 
    which are known to be abstract \emph{bicategories of relations} in the 
    sense of \cite{carboniwalters1987} which are recovered as examples of 
    locally posetal cartesian equipments in \cite{lambert2022b}.
    \item Such a general framework could afford applications to, and embrace 
    different approaches to, \emph{formal epistemology}. For example, the 
    geometric semantics of knowledge and belief in \cite{baltagetal2015} 
    plausibly fall under this picture. Common knowledge \cite{lehmann1984} is 
    potentially modeled using ordinary morphisms of the underlying category of 
    a double category which model subcollections of agents. This can 
    potentially be generalized to give relationships between modal models where 
    models are proarrows and ordinary arrows induce transition morphims from 
    the restriction and extension constructions of the ambient equipment. 
    Likewise, synthetic probability theory \cite{fritz2019} has not yet been 
    phrased in double-categorical terms, but the \emph{Markov categories} of 
    that subject bear striking resemblences to cartesian bicategories. 
    Commutative comonoids are an example. In general double categories allow 
    for the unified description of deterministic and non-deterministic 
    processes. They are thus a natural forum to try to recast the results and 
    methods of that subject. A satisfactory double-categorical synthetic 
    probability theory could almost certainly be based on some version of the 
    axioms for a cartesian equipment. Thus, we envision a further application of such a framekwork for formal epistemology to a description of probabilistic epistemic logic \cite{demey2015}.
    \item The philosophical import of the results of the paper and especially 
    those in the last section could certainly be studied further. A more 
    detailed analysis of the possibilism/actualism question seems in order. 
    Additionally, questions related to \emph{necessary entities} and 
    \emph{essential properties} \cite{parker1976}, may profitably be modeled 
    using the formal methods introduced in this paper.
\end{enumerate}

\bibliographystyle{alpha}
\bibliography{refs}

\end{document}